\documentclass{article}

\usepackage{amssymb}
\usepackage{mathtools}
\usepackage{mathdots}

\usepackage{amsmath}
\usepackage{amscd}
\usepackage{amsthm}
\usepackage{graphicx}
\usepackage[all]{xy}

\usepackage{color}          
\usepackage{epsfig}

\usepackage{lineno}

\newcommand{\C}{\mathbb {C}}

\newcommand{\Addresses}{{% additional braces for segregating \footnotesize
  \bigskip
  \footnotesize

  C.~Cabrera  \textsc{Unidad Cuernavaca del Instituto de Matem\'aticas. UNAM, M\'exico}\par\nopagebreak
  \textit{E-mail:}  \texttt{carloscabrerao@im.unam.mx}

  \medskip

  P.~Makienko  \textsc{Unidad Cuernavaca del Instituto de Matem\'aticas. UNAM, M\'exico}\par\nopagebreak
  \textit{E-mail:} \texttt{makienko@im.unam.mx}}}

\newtheorem{theorem}{Theorem}
\newtheorem{lemma}[theorem]{Lemma}
\newtheorem{corollary}[theorem]{Corollary}
\newtheorem{proposition}[theorem]{Proposition}\newtheorem*{conjecture}{Conjecture}
\theoremstyle{definition}
\newtheorem*{definition}{Definition}
\newtheorem{example}{Example}

\title{Amenability and measure of maximal entropy for semigroups of rational maps.}
\author{Carlos Cabrera and Peter Makienko}

\begin{document}

\maketitle
\footnotetext{This work was partially supported by PAPIIT IN106719 and CONACYT CB15/255633.

}
 \begin{abstract}

In this article we discuss relations between algebraic and dynamical properties of non-cyclic semigroups of rational maps. 

\end{abstract}

\section{Introduction}

In a series of works J. Ritt (see \cite{Rittper} and \cite{RittEquiv}) studied non-trivial relations and functional equations on the semigroup of rational maps. Specifically, Ritt was interested in the solution of equations of the following type $$A\circ B=C\circ D$$ where $A,B,C$ and $D$ are rational maps.  Ritt gave solutions to these equations for polynomials but there were obstacles in the case of rational maps.  Ritt's theory  for rational maps is still under investigation and presents many open questions. The paper  \cite{PakovichComRat} contains a short survey on the modern treatment in this area as well of an ample list of references. In particular, the references provided whitin \cite{PakovichComRat} also include a series of Pakovich's works on recent developments of  Ritt's theorems for rational maps.   

In the paper \cite{GhiTucZieve}, D. Ghioca, T. J. Tucker and M. E. Zieve proved the following interesting result:

\textit{
If for polynomials $P$ and $Q$ there exists a point $z_0\in \overline{\C}$ such that the intersection of the forward orbits of $z_0$, with respect to $P$ and $Q$, is an infinite set then 
$P$ and $Q$ share a common iterate. That is, there are natural numbers $n,m$
 such that $P^n=Q^m.$ }
 
 In other words, the dynamical intersection property implies an algebraic equation on $P$ and $Q.$
 
 Our first theorem generalizes the theorem above.  Recall that a polynomial (rational map) $Q$ is \textit{exceptional} if either $deg(Q)=1$ or  $Q$ is affinely (M\"obius) conjugated either to $z^n$ (with $n\in \mathbb{Z}$) or to a Chebyshev polynomial (or a Latt\`es map). We say that a family of polynomials (rational maps) $\mathcal{F}$ is \textit{non-exceptional} if $\mathcal{F}$ contains a non-exceptional polynomial (rational map).  Given a family of maps $\mathcal{F}$ we denote by $S(\mathcal{F})$ the semigroup generated by the family $\mathcal{F}.$
 
 \begin{theorem}\label{th.polynomialsemb}
 Given a finite non-exceptional family $\mathcal{F}$ of non-injective  polynomials. Then,  the following conditions are equivalent. 
 
 \begin{enumerate}
  \item For every pair $P$, $Q \in \mathcal{F}$ there exists a point $z_0\in \overline{\C}$ such that $$\#(\{\bigcup P^j(z_0) \cap \bigcup Q^k(z_0)\})=\infty.$$
  
    \item For every pair $P,Q\in S(\mathcal{F})$ there are integers $m,n$ such that $P^m=Q^n.$
    \item The semigroup $S(\mathcal{F})$ is amenable with $RIM(S(\mathcal{F}))\subset LIM(S(\mathcal{F})).$ Where $RIM(S(\mathcal{F}))$ and $LIM(S(\mathcal{F}))$ stands for the sets of right and left invariant means respectively.
    \item $S(\mathcal{F})$ is embeddable into a virtually cyclic group.
 \end{enumerate}

 \end{theorem}
Recall that a virtually cyclic group is a group containing a cyclic subgroup of finite index.  As shown in Example \ref{example} below, a semigroup of polynomials may be embeddable into a virtually cyclic group  but not into a metacyclic group, those are virtually cyclic groups for which the cyclic subgroup is  normal.

 For families of rational maps we have the following. 
 
 \begin{theorem}\label{th.mainA}
  Let $\mathcal{F}$ be a non-exceptional collection of non-injective rational maps containing an element not M\"obius conjugated to a polynomial. Then the 
  following conditions are equivalent.
  \begin{enumerate}
   \item The semigroup $S(\mathcal{F})$  is right amenable and for every pair $P,Q\in S(\mathcal{F})$ there is a point $z_0\in \overline{\C}$ such that $$\#\{\bigcup P^j(z_0) \cap \bigcup Q^k(z_0)\}=\infty.$$
   \item For every pair $P,Q\in S(\mathcal{F})$ there exist natural numbers $m,n$ such that $P^m=Q^n.$
   \item The semigroup $S(\mathcal{F})$ is right amenable and embeddable into a group.
   \item The semigroup $S(\mathcal{F})$ is $\rho$-right-amenable, where $\rho$ denotes the Lyubich representation (definitions below), and for every pair $P,Q\in \mathcal{F}$ there is a point $z_0\in \overline{\C}$ such that $$\#(\{\bigcup P^j(z_0) \cap \bigcup Q^k(z_0)\})=\infty.$$
   
  \end{enumerate}

 \end{theorem}

 If a semigroup of rational maps is finitely generated and  satisfies the condition (3) of Theorem \ref{th.mainA}, then the associated group is virtually cyclic.

The following two theorems  describe the right-amenable semigroups of rational maps. 
 
 \begin{theorem}\label{th.rhorightam}
Let $\rho$ be the Lyubich representation. Assume $S$ is a $\rho$-right-amenable semigroup of rational maps containing a non-exceptional rational map which is not M\"obius conjugated to a polynomial,  then the following statements hold true. 
  
  \begin{enumerate} 
   \item There exists a probability measure $\mu$ invariant under $S$.
   \item If $P\in S$ and $deg(P)>1$ then $\mu$ is the measure of maximal entropy of $P$.
  \end{enumerate}

 \end{theorem}
Since right amenability implies $\rho$-right-amenability for every bounded representation $\rho$ (definitions and discussion in the next section), then Theorem \ref{th.rhorightam} holds for  right-amenable semigroups of rational maps. 

 For polynomials we have the following result.
   
 \begin{theorem}\label{th.rholyub}
  Given a finite non-exceptional collection of  polynomials $\mathcal{F}$, the following conditions are equivalent.
  \begin{enumerate}
   \item The semigroup $S(\mathcal{F})$ is $\rho$-right-amenable for the Lyubich representation $\rho.$
   \item The semigroup $S(\mathcal{F})$ is right amenable.
\item There exists a probability measure $\mu$ invariant under $S(\mathcal{F})$ which coincides with the measure of maximal entropy for every element in $S(\mathcal{F})$.
  \end{enumerate}

 \end{theorem}

 The equivalence of (1) and (2) is rare even among groups. According to M. Day (see Theorem 2 in  \cite{Daymeans}) a semigroup $S$ is right amenable if and only if $S$ is $\rho$-right amenable for every bounded representation $\rho.$
 
 Hence the theorem above presents an interest from this point of view.
 
 For rational maps, Theorem \ref{th.rholyub}  is still an open question. We propose the following:

 \begin{conjecture} If a semigroup $S$ of rational maps  admits an invariant probability meausre which is  the measure of maximal entropy of every non-injective element of $S$, then $S$ is a right-amenable semigroup.
 \end{conjecture}
 
 Theorem \ref{th.rholyub} gives a partial answer to this conjecture (see also Theorem \ref{th.entropy} and Corollary \ref{cor.ramenable}).
 
 In fact, Theorems \ref{th.polynomialsemb}-\ref{th.rholyub} hold in more general settings, see the discussion in  Section 3.

In the last section we discuss amenability properties for another important representation in holomorphic dynamics, namely Ruelle representation. The Ruelle representation is closely related to quasiconformal deformations of rational maps. In Proposition \ref{pr.hypstruc} and Proposition \ref{pr.collstruc}
we show that a right amenable semigroup $S$ is quasiconformally deformable whenever $S$ contains a hyperbolic element which admits a non-trivial quasiconformal deformation.  Even more, a semigroup $S$ of rational maps is structurally stable whenever $S$ satisfies the Levin relations, is finitely generated  and contains a structurally stable element. 

To every rational map $R$  of degree at least $2$, we associate a right-amenable finitely generated semigroup of operators $D(R)$ acting on the space $L_1(A)$, for every Lebesgue measurable subset $A$  completely invariant with respect to $R.$ In Proposition \ref{pr.BeltDR} we observe that if $A$ does not possess a Beltrami differential, invariant under $R$, then the action of $D(R)$ on $L_1(A)$ is left amenable. The main theorem of the last section is the following. 
 \begin{theorem}\label{th.Ruelle} If $R$ is a rational map and assume that the action of $D(R)$ in $L_1(\overline{\C})$ is left-amenable. Then the following conditions are equivalent:
 \begin{enumerate}
  \item The Ruelle operator $R_*$ has non-zero fixed points in $L_1(\overline{\C})$.
  \item $R$ is M\"obius conjugated to a flexible Latt\'{e}s map. 
 \end{enumerate}
\end{theorem}
 
 The dynamics of non-cyclic semigroups of rational maps initiated by A. Hinkkanen and G. Martin in \cite{HinkkMartSemigroups} is now an active area of research in holomorphic dynamics. Yet another approach is presented in \cite{CaMakPl} and \cite{EreEnd}.
 
 In particular, in \cite{HinkkMartSemigroups} the authors adapt the Klein-Maskit combination theorem to construct free semigroups of rational maps.  The arguments in \cite{HinkkMartSemigroups}, allow to show the following statement. 
 
 \textit{If the polynomials $P, Q$ have mutually disjoint filled Julia sets, then there are integers  $m,n>0$ such that the semigroup $\langle P^m, Q^n\rangle$ is free.}
 
 So the semigroup $\langle P, Q \rangle$ contains a free two-generated subsemigroup, which is neither right nor left amenable. This observation is another motivation for considering amenability conditions.  
 
We have the following conjecture:
\begin{conjecture}
A finitely generated semigroup $S$ not containing a rank 2 free subsemigroup admits an invariant probability meausre which is the measure of maximal entropy of every non-injective element of $S$.
\end{conjecture}

 In this work, we will use standard notations and concepts from holomorphic dynamics which can be found, for instance, in \cite{L}.
 
 Let us describe an outline of the article as follows,  in section 2, we give some basic facts and notions of amenability of semigroups, holomorphic dynamics, measure or maximal entropy and Ruelle and Lyubich representations. We also introduce three intersection properties for semigroups of rational maps, the comparation between them is one of the main motivations of this work leading to Theorem \ref{th.polynomialsemb}, this theorem is proved in a slightly more general version in section 3 as Theorem \ref{th.familypolynomials}. 
 
 Theorem \ref{th.mainA} is the version of Theorem \ref{th.polynomialsemb} for arbitrarily semigroups of rational maps containing an non-exceptional map not M\"obius conjugated to a polynomial. Theorem \ref{th.mainA} is proved in section 3 as Theorem \ref{th.Intersection}.  
 
 The proofs of Theorem \ref{th.polynomialsemb} and Theorem \ref{th.mainA}  are based upon Theorem \ref{th.rhorightam} and Theorem \ref{th.rholyub}, respectively, which describe (and characterize for polynomials) right amenable semigroups of rational maps. These theorems  are proven in section 3 as 
 Theorem \ref{nonatomicLyubich} and Corollary \ref{cor.ramenability}, respectively. The gap between the stronger case of polynomials and arbitrary semigroups rational maps is contained in the conjectures above. So in section 3, we consequently develop all ingredients and combine them toghether to form the proofs. 
 
  Presenting independent interest, Proposition \ref{pr.33} and  Corolary \ref{cor.34} describe quotients and representations of rational semigroups for an equivalence relation motivated by Proposition  \ref{pr.Ritt}. Also,  Theorem \ref{th.cmiterate} and Theorem \ref{th.RLamenable} can be regarded as a characterization of amenable semigroups of polynomials.  
 
Finally, last section is devoted to the proof of  Theorem \ref{th.Ruelle} which is presented as Theorem \ref{th.fixed}.

While this paper was under revision, we learned that some results of this paper were generalized in \cite{pakovich2020amenable} and later the conjecture about rank 2 free semigroups was answered in \cite{TitsTucker}.

 \textbf{Acknowledgement.} The authors would like to thank F. Pakovich for useful discussions, and to the referee for useful remarks and suggestions. 
 
\section{Preliminaries}

\subsection{Semigroup amenability}\label{sect.Semigroup}

Let $S$ be a semigroup and let $L_\infty(S)$ be the linear space of bounded complex valued functions equipped with the supremum norm. A continuous linear functional  $M$ on $L_\infty(S)$ is called a \textit{mean} if $M$ satisfies the following properties:

\begin{enumerate}
 \item $M$ is positive,  that is, if $\phi\in L_\infty(S)$ and $\phi\geq 0$ then $M(\phi)\geq 0.$
 \item $\|M\|=M(\chi_S)=1$, where $\chi_S$ denotes the characteristic function of $S.$
\end{enumerate}
The right and left actions of $S$ onto itself generate right and left actions on the space $L_\infty(S)$ given by the formulas 

$$r_s(\phi)(x)=\phi(x s)$$
$$l_s(\phi)(x)=\phi(s x)$$
for every $s,x\in S$, and $\phi\in L_\infty(S), $ respectively.

These actions induce right and left representations of the semigroup $S$ into the semigroup $End(L_\infty(S))$ of linear continuous  endomorphisms of $L_\infty(S)$ given by $s\mapsto r_s$ and $s\mapsto l_s.$

The semigroup $S$ is called \textit{right amenable}, or shortly an RA-semigroup, if there exists a mean  which is invariant for the right action 
of $S$ on $L_\infty(S)$ that is  $M(r_s(\phi))=M(\phi)$ for every $s\in S$ and $\phi\in L_\infty(S)$. 
We denote by $RIM(S)$ the set of all right invariant means on the 
semigroup $S.$ Note that $RIM(S)$ is a convex, closed subset of 
$L^*_\infty(S)$ which does not contains the $0$ functional. Even more $RIM(S)$ is compact in the $*$-weak topology. 

Analogously, $S$ is called \textit{left amenable}, or an LA semigroup for short, if there is a mean invariant under the left action. We denote by $LIM(S)$ the set of all left invariant means.  

Finally, the semigroup $S$ is called an \textit{amenable semigroup}, if $RIM(S)\cap LIM(S)\neq 
\emptyset.$   

Let us mention some basic facts about amenable semigroups.  Further details and proofs may be found in the papers of M. Day  \cite{DayAmenable} and \cite{DaySemigroups}. 

\medskip

\textbf{Basic Facts on Amenability.}
\begin{enumerate}
 \item Every abelian semigroup is amenable.
 \item Every finite group is amenable and not every finite semigroup is amenable. For example,  the finite semigroup $\langle a,b: ab=a^2=a, ba=b^2=b\rangle$
 is not an LA-semigroup.
 \item Every semigroup is a subsemigroup of an amenable semigroup. But, for groups every 
  subgroup of an amenable group is amenable.
 \item Given a semigroup $S$, let us consider the antiproduct $*$ on the set $S$ defined by $a*b=b 
 a.$  The set $S$ equipped with the product $*$ is a semigroup anti-isomorphic to $S$ and thus the space $RIM(S)$ coincides with  $LIM(S,*)$. In the same way, the left action of  $S$ on $L_\infty(S)$ is the right action of $(S,*)$ on $L_\infty(S).$

 \item Let $M$ be an element of either $LIM(S)$ or $RIM(S).$ If $M(\chi_{S_0})>0$ for a subsemigroup $S_0<S$. Then $S_0$ is itself either an LA or an RA-semigroup, respectively. 
 \item Let $\phi:S_1\rightarrow S_2$ be an epimorphism of semigroups, then  $S_2$ is either $RA$ or $LA$ whenever $S_1$ is $RA$ or $LA$, respectively.
 \end{enumerate}

Let us consider a weaker version of amenability, namely $\rho$-\textit{amenability} which was introduced by M. Day in \cite{DaySemigroups}.

First, we say that a proper right (left) $S$-invariant  subspace $X\subset L_\infty(S)$ is called either \textit{right} or \textit{left amenable}, again RA or LA for short,  if $X$ contains constant functions and there exists a mean $M$ such that, when $M$ is restricted to $X$, it induces a functional which is invariant for either  the right or left actions of $S$ on $X$, respectively. In other words $X$ is invariant and admits an invariant functional for the associated action. Note that every semigroup admits an amenable subspace, for example,   the subspace  of constant functions is always amenable. 

Now, let $\rho$ be a bounded homomorphism from $S$ into $End(B)$, where $B$ is a Banach space, and $End(B)$ is the semigroup of continuous linear endomorphisms. Given a pair 
$(b,b^*)\in B\times B^*$ consider the function $f_{(b,b^*)}\in L_\infty(S)$ given by
$$f_{(b,b^*)}(s)=b^*(\rho(s)(b)).$$  Let $Y_\rho\subset L_\infty(S)$, be the closure of the linear span of the family of functions $\{ f_{(b, b^*)} \}$ for all pairs $(b,b^*)\in B\times B^*.$ Finally let $X_\rho$ be the space generated by $Y_\rho$ and the constant functions. Note that $X_\rho$ and $Y_\rho$ are both right and left invariant. 
\begin{definition}
We will say that $\rho$ is either \textit{RA or LA} whenever $X_\rho$ is either a right or left amenable subspace of $L_\infty(S)$, respectively.  Also we will say that $S$ is $\rho$-RA or $\rho$-LA whenever $\rho$ has the respective property. Equivalently, that the $\rho$-action of $S$ on $B$ is either RA or LA, respectively.
\end{definition}

To show that, in general, amenability is different from $\rho$-amenability let us recall Day's theorem  (see Theorem 2 in  \cite{Daymeans}).

\begin{theorem}[Day Theorem]\label{th.Day} A  semigroup $S$ is either RA or LA if and only if $S$ is either $\rho$-RA or $\rho$-LA, respectively, for every bounded representation $\rho$.
 
\end{theorem}

Roughly speaking, the existence of an invariant functional on a proper subspace does not always implies the existence of an invariant functional on $L_\infty(S)$. 

For example consider a free group $G$ which is neither RA nor LA. Let $h$ be a homomorphism from $G$ onto a non trivial abelian group $\Gamma$, then  the space $h^*(L_\infty(\Gamma))\subset L_\infty(G)$ is amenable,  where $h^*(\phi)=\phi\circ h$ is the pull-back operator. We do not know examples of semigroups which are neither $\rho$-RA nor $\rho$-LA for every bounded representation $\rho$,  even in the case when the associated Banach space is infinitely dimensional.  

\subsection{Maximal entropy and representations}

In this article we consider two important representations of 
semigroups of rational maps. Namely, Lyubich and Ruelle representations, these are push-forward actions of rational maps on the spaces $C(\overline{\C})$ and $L_1(\overline{\C})$ of continuous and Lebesgue integrable functions on the Riemann sphere $\overline{\C}$, respectively.

Let us first discuss Lyubich representation. Every rational map $R$ induces an  operator given by $$L_R(\phi)(y)=\frac{1}{deg(R)}\sum_{R(x)=y} \phi(x)$$ where the sum is taken with multiplicities. The operator $L_R$ is a continuous endomorphism of $C(\overline{\C})$ with the unit norm. The operator $L_R$ was firstly considered by M. Lyubich in \cite{LyubichEntropy}, we call $L_R$ the \textit{Lyubich operator} of the rational map $R.$  Now we reformulate the main results of \cite{LyubichErgodicTheory} as follows:
 
\begin{theorem}\label{th.Lyubich}
 For every rational map $R$ with $deg(R)>1$ there exist an invariant non-atomic probability measure $\mu_R$ which represents an invariant functional on $C(\C)$ with respect to the Lyubich operator $L_R$. The measure $\mu_R$  is of maximal entropy, ergodic and unique in the following sense: if an $L_R$-invariant functional is generated by a non-atomic measure $\nu$, then $\nu$ is a multiple of $\mu_R$.  
 \end{theorem}

The support of $\mu_R$ coincides with the set $J(R)$, the \textit{Julia set} of $R$.
Observe that the Lyubich operator is well defined for every  bran\-ched self-covering of the Riemann sphere of finite degree. 

\begin{definition}
 Let $f:\overline{\C}\rightarrow \overline{\C}$ be a branched covering of finite degree. We call the correspondence $\rho:f\mapsto L_f$ the \textit{Lyubich representation}.
\end{definition}

Lyubich representation gives a homomorphism from the whole semigroup of finite degree branched self-coverings of $\overline{\C}$ into $End(C(\overline{\C}) )$, the semigroup of continuous linear endomorphisms of the space of continuous functions on $\overline{\C}$.

Note that the uniqueness statement in Theorem \ref{th.Lyubich} is, in general, false for  non-holomorphic branched self-coverings (see discussion after Theorem \ref{th.RABranched}).
We call a  complex valued measure $\nu$ a \textit{Lyubich measure} for a semigroup $S$ generated by a collection of finite degree branched self-coverings of $\overline{\C}$ whenever $\nu$ induces an $L_f$-invariant functional for every $f\in S.$

Now, let us discuss the Ruelle representation of rational maps. 

\begin{definition}
 Let $R$ be a rational map, then the operator $$R_*(\phi)(y)=\sum_{R(x)=y}\frac{\phi(x)}{(R'(x))^2}$$ is called the \textit{Ruelle transfer operator} or the  \textit{Ruelle operator} .  
\end{definition}

Ruelle operator acts on the space $L_1(\overline{\C})$ with $\|R_*\|\leq 1$.  The operator $B_R(\phi)=\phi(R)\frac{\overline{R'}}{R'}$ is called the \textit{Beltrami operator}, which is a continuous endomorphism of $L_\infty(\overline{\C})$ with unitary norm. The space $Fix(B_R)$ of fixed points of $B_R$ is called the space of \textit{invariant Beltrami differentials} of $R$. In other words, the form 
$\phi(z)\frac{\overline{\partial}z}{\partial z}$ is invariant under the pull-back action of $R$ whenever $\phi\in Fix(B_R).$ By Ahlfors-Bers theorem, the space of invariant Beltrami differentials generates all quasiconformal deformations of the map $R.$

The relevance of Ruelle operator comes from the following lemma (see  for example \cite{CMFixed} and \cite{MakRuelle}). 

\begin{lemma} The Beltrami operator $B_R$ is dual to the Ruelle operator $R_*$.
\end{lemma}

Let us note that both the Beltrami and the Ruelle operators can be extended to almost everywhere differentiable branched self-coverings of the Riemann sphere. 

\subsection{Relations and functional equations on rational maps}

 The following theorem was proven by Ritt in \cite{Rittper} and completed by Eremenko in \cite{EreFunc}.

\begin{theorem}\label{th.RittEremenko}
 Let $S\subset Rat$ be an abelian semigroup of rational maps. Assume that $S$ contains a non-exceptional element $R$ with $deg(R)\geq 2$. Then for every pair of elements $P,Q\in S$ with $deg(P),deg(Q)\geq 2$ there are numbers $m,n$ such that $P^m=Q^n.$
\end{theorem}

\begin{definition} 
 We say that the rational maps $Q, R$ satisfy the \textit{Levin relations} if $$Q\circ R=Q\circ Q$$ and $$R\circ Q=R\circ R.$$
\end{definition}
The following theorem is proved in \cite{LevinPrytyckiRel} and \cite{Levinrelations}, we present it as formulated by H. Ye in \cite{Ye}.

\begin{theorem}\label{th.Levin-Przyticki}
Two non-exceptional rational maps $Q$ and $R$ share the same measure of maximal entropy if and only if there are numbers $m, n$ such that $Q^m$ and $R^n$ satisfy  the Levin relations.
\end{theorem}

The following theorem is a consequence of Ritt's results given in \cite{RittEquiv}. 

\begin{proposition}[Ritt]\label{pr.Ritt}
 Let $F, A,B$ be rational maps satisfying the equation $$FA=FB,$$ 
 then either $deg(F)>deg(A)=deg(B)$ or $A$ and $B$ share a common right factor, that is, there are rational  maps $X, Y$ and $Z$ such that $$A=X\circ Z$$ and $$B=Y\circ Z.$$ 
\end{proposition}

By Proposition \ref{pr.Ritt}, if $Q$ and $R$ satisfy the Levin relations, then $Q$ and $R$ share a right common factor. Moreover, if either $Q$ or $R$ is an indecomposable rational map then the rational maps $X,$ and $Y$ given in  Proposition \ref{pr.Ritt} must be M\"obius transformations.  Recall that a map $R$ is called \textit{indecomposable} if whenever we have  the equation $R=P\circ T$ then one of the factors, either $P$ or $T$ must be a M\"obius transformation.

Also note that the relations given in Proposition \ref{pr.Ritt} pose an obstacle to the left cancellation property (definitions and discussions are given below). 

\subsection{Intersection properties} 

Now let us introduce three intersection properties which will be discussed in this work.

\begin{definition}[Dynamical intersection property] Let $DIP\subset Rat\times Rat$ be the set consisting of the pairs of rational maps 
$Q$, $R$ for which there exists a point $z_0\in \overline{\C}$ with 
 $$ \#\{ \mathcal{O}_Q(z_0) \cap \mathcal{O}_R(z_0)\}=\infty,$$ where $\mathcal{O}_R(z)=\bigcup_{n\geq 0} R^n(z)$ denotes the forward orbit of $z$.
\end{definition}

\begin{definition}[Algebraic intersection property]
 Let $AIP\subset Rat\times Rat$ be the set of all pairs $(Q,R)$  sharing a common iteration.   
\end{definition}
Note that $AIP\subset DIP$. 
\begin{definition}[Ideal intersection property]
 The semigroup $S$ satisfies the \textit{left or right ideal intersection} property whenever every pair of principal left or right ideals $I,J$ in $S$ have non-empty intersection.
\end{definition}

The last property  is closely related to the problem of embedding a semigroup into a group. That is to specify under what circumstances a given semigroup $S$ is ``half'' of a group.

Let $\Gamma$ be a countable  group with a minimal set of generators $\langle \gamma_1,...,\gamma_n,...\rangle$, consider the subset $\Gamma_+$ of all words in the alphabet $\{\gamma_1,...\gamma_n,...\}$. Then $\Gamma_+$ forms a countable semigroup which is called the \textit{positive part} of $\Gamma.$ Note that 
$\Gamma$ is generated by $\Gamma_+$ and $(\Gamma_+)^{-1}$. A countable semigroup $S$ is \textit{embeddable into a group} if $S$ is isomorphic to the positive part of a group.

Recall that a semigroup $S$ is \textit{left cancellative} if for $a,b,c\in S$
the equation $ca=cb$ implies $a=b.$ An analogous definition applies for a \textit{right cancellative} semigroup $S$. For example, every semigroup generated by a set of surjective endomorphisms of a set $A$ is always right cancellative. In particular, every semigroup of rational maps is always right cancellative.

If $S$ is both left and right cancellative, then $S$ is called a \textit{cancellative} semigroup. For instance, any finitely generated free semigroup $S$ is cancellative and, even more, $S$ is embeddable into a finitely generated free group. 

The following theorem due to O. Ore  provides sufficient conditions for a semigroup to be embeddable into a group (see \cite{OreSem}).  

\begin{theorem}[Ore Theorem]\label{th.Ore}
 Let $S$ be a cancellative semigroup, then $S$ is embeddable into a group whenever $S$ satisfies either the left or right ideal intersection property.
\end{theorem}

In fact, Ore Theorem  does not need the countability condition. As a consequence of Ore Theorem we have that every abelian semigroup  $S$ is embeddable into a group if and only if  $S$ is cancellative. Hence every abelian semigroup of rational maps is embeddable into a group.

 In order to apply Ore Theorem, we need either the right or the left ideal intersection property which is known for RA semigroups (see for example \cite{Klawe}). For sake of completeness we include it in the following lemma.
 
 \begin{lemma}\label{lm.LIP}
  If $S$ is an RA semigroup then $S$ satisfies the left ideal intersection property.
 \end{lemma}
\begin{proof}
   If for $P, Q\in S$ we have  $SP\cap SQ= \emptyset$ then for every $r$-mean $\nu$ we have 
 $$\nu(\chi_S)\geq \nu(\chi_{SP}+\chi_{SQ})= \nu(\chi_{SP})+\nu(\chi_{SQ})=2\nu(\chi_S)$$ which is a contradiction.
\end{proof}
The following corollary is an immediate consequence of Ore Theorem and Lemma \ref{lm.LIP}.
\begin{corollary}\label{cr.ratemb}
 A semigroup $S$ of rational maps is embeddable into group whenever $S$ is left cancellative and RA.
\end{corollary}

\subsection{Ergodic actions}

Given an operator $T$ on a Banach space $X$,  the 
$n$-\textit{Ces\`aro averages} of $T$ are the operators $A_n(T)$ defined for 
$x\in 
X$ 
by $$A_n(T)(x)=\frac{1}{n}\sum_{i=0}^{n-1}T^i(x).$$

An operator $T$  on a Banach space $X$ is called  \textit{mean-ergodic} if $T$ 
is power-bounded, that is, $\|T^n\|\leq M$ for some number $M$ 
independent of $n$, and the Ces\`aro averages $A_n(T)(x)$ converge in  norm for 
every $x\in X$.

The following fact can be found, for example, 
in Krengel's book \cite{Krengel}.

 \noindent \textbf{Separation principle}. \textit{An operator $T$ is mean-ergodic if and only if $T$
satisfies the principle of separation of fixed points:}

\textit{If $x^*$ is a fixed point of 
$T^*$, where $T^*$ denotes the dual operator of $T$, then there exists  $y\in X$ a fixed point of $T$ such that  $\langle x^*,y\rangle \neq 0.$}

Recall that an operator $T$ acting on a Banach space $\mathcal{B}$ is called \textit{weakly almost periodic} if $\{T^n(f)\}$ is weakly sequentially precompact for every $f\in \mathcal{B}.$ The following theorem is due to I. Kornfeld and M. Lin \cite{KornfeldLin}.
\begin{theorem}\label{th.Lin}
 Let $T$ be a positive operator with $||T||\leq 1$  acting on $L_1(X,\mu)$ space. The operator $T$ is weakly almost periodic if and only if $T$ is mean-ergodic.  
\end{theorem}

\section{Lyubich representation}

We start with the following theorem.
\begin{theorem}\label{th.RABranched}
 Let $S$ be a semigroup of branched self-coverings of the sphere. If the Lyubich representation of $S$ is right amenable, then there exists a Lyubich probability measure for $S$. 
\end{theorem}

\begin{proof} Let $\rho:S\rightarrow End(C(\overline{\C}))$ be 
the Lyubich representation.
 Let $\sigma$ be a probability measure. Let  $H:C(\bar{\C})\rightarrow L_\infty(S)$ be 
 the map defined for $s\in S$ by $$H(\phi)(s)=\int_{\overline{\C}} \rho(s)(\phi(z))d \sigma(z).$$
 Since the characteristic function $\chi_{\overline{\C}}$ is a fixed element for every  Lyubich operator $\rho(s)$, then the closure of the image of 
$H$ is a subspace $L_\infty(S)$ containing 
the constant functions on $S.$ The space $X=\overline{im (H)}$ is invariant under the right action of $S.$ By assumption $X$ admits a non-zero  $r$-invariant mean $L$, then the  functional $\ell$ given by $$\ell(\phi)=L(H(\phi))$$ is continuous and positive on $C(\overline{\C})$. Let
us show that $\ell$ is invariant with respect to  $\rho(S).$ 

Indeed, for  $t\in S$
$$r_t(H(\phi))(s)=H(\phi)(st)=\int \rho(st)(\phi)d\sigma$$ $$=\int \rho(s)( \rho(t) (\phi))d\sigma=H(\rho(t)(\phi))(s).$$ Since $L$ is $r$-invariant then $$\ell(\phi)=L(H(\phi))=L(r_t(H( \phi)))$$
$$=L(H(\rho(t)(\phi)))=\ell(\rho(t)( \phi)).$$

By the Riesz representation theorem there exists a probability measure $\mu$ satisfying $\ell(\phi)=\int \phi d\mu $, since $\ell$ is $\rho(t)$ invariant then  $\mu$ is a Lyubich measure.\end{proof}

 The measure $\mu$ depends on the choice of the measure $\sigma$. For instance, consider  $\sigma=\delta_{z_0}$ the delta measure on a suitable point $z_0\in \overline{\C}$. If the cardinality of $\mathcal{O}_-(S)(z_0)$ is finite  then $\mu$ is atomic and, in fact, is a linear 
combination of delta measures based on $\overline{\mathcal{O}_-(S)(z_0)},$ where $$\mathcal{O}_-(S)(z_0)=\bigcup_{n\geq 0,s\in S} s^{-n}(z_0).$$  If $S<Pol$ is a polynomial semigroup then, choosing $z_0=\infty$, the measure $\delta_{\infty}$ is an atomic Lyubich measure.

If $\overline{\mathcal{O}_-(S)(z_0)}$ is infinite then $\mu$ may be non-atomic, as in the case of  cyclic semigroups of rational maps.

Even more, for  semigroups of non-holomorphic branched coverings of the sphere a non-atomic Lyubich measure may not be unique even for cyclic semigroups.  For example, if $f$ is a formal mating of two polynomials, say $P$ and $Q$, then the conformal copies of the measures of maximal entropy for $P$, $Q$ and $z^{deg P}$  generate a three dimensional space of Lyubich measures for $f$.  One can use the tuning procedure to construct a map with two dimensional space of Lyubich measures, hence by repeating the procedures of mating and tuning we can produce a multidimensional space of Lyubich measures. The definition of tuning and matings was given by A. Douady in section 2 and section 3 of \cite{Douady}.

From Theorem \ref{th.RABranched} we have the following corollary. 

\begin{corollary}
Let $S$ be an RA  semigroup of branched self-coverings  of the Riemann sphere. Then there exists a Lyubich probability measure for $S$. 
\end{corollary}

For semigroups of rational maps we have  Theorem \ref{th.rhorightam} which gives a stronger conclusion. For this we need the following.

\begin{definition}
We say that a collection of rational maps $\mathcal{F}$ is called \textit{admissable} if it contains a non exceptional element $g$ and either 
\begin{itemize}
 \item  $g$ is not M\"obius conjugated to a polynomial, or 
 
 \item $g$ is conjugated to a polynomial but there exists another element $g_1\in \mathcal{F}$ such that there is no M\"obius map $\gamma$ simultaneously conjugating $g$ and $g_1$ to  polynomials, or  
 
 \item $\mathcal{F}$ consists of rational maps  simultaneously M\"obius conjugated to a family of polynomials
 $\mathcal{F}'$ and there exists a disk $D_\mathcal{F}$ centered at $\infty$ such that $P(D_\mathcal{F})\subset D_\mathcal{F}$ for every $P\in \mathcal{F}'$.
\end{itemize}

\end{definition}

In other words, a family $\mathcal{F}$ containing a non exceptional map is admissable if either there is no point $z_0\in \overline{\C}$ with $R^{-1}(z_0)=\{z_0\}$ for every $R\in \mathcal{F}$, or, otherwise, there exists a common $R$-invariant topological disk around such $z_0$ for every $R\in \mathcal{F}$. 

Every finite collection of non-linear polynomials containing a non-exceptional polynomial is admissable, even more every collection of monic non-linear polynomials with bounded coefficients and containing a non-exceptional polynomial  is admissable. Finally,  a collection $\mathcal{F}$ is admissable if and only if the semigroup $S(\mathcal{F})$ is also  admissable.  Indeed, if $\mathcal{F}$ is admissable, then $S(\mathcal{F})$ is also admissable. Reciprocally, if $S(\mathcal{F})$ is admissable but $\mathcal{F}$ is not then $\mathcal{F}$ is M\"obius conjugated to polynomials and there is no disk $D$ such that $P(D)\subset D$ for all $P\in S(\mathcal{F})$ which contradicts $S(\mathcal{F})$ is admissable. 

We also need the following lemma.

 \begin{lemma}\label{lm.atomic}
Let $R$ be a non-injective rational map and $\nu$ be a probability Lyubich measure for $R$. Then $\nu=s \mu_R+(1-s) \omega$ with $s\in [0,1]$ where $\mu_R$ is the measure of maximal entropy and either $\omega=\delta_a$ with $R^{-1}(a)=a$ or $\omega=t\delta_a+(1-t)\delta_b$ with $t\in [0,1]$ and $R^{-1}(\{a,b\})=\{a,b\}$. 
 \end{lemma}
 In other words, an $L_R$-invariant functional is presented by an atomic measure $\nu$ if and only if either $R$ or $R^2$ is M\"obius conjugated to a polynomial and the support of $\nu$ contains the point associated to $\infty$ as an atom of $\nu$. 
 
\begin{proof}
 Let us observe that the support of a measure $\nu$ is a completely invariant closed set. Indeed,   let $A=supp(\nu)$ and $\chi_A$ be its characteristic function. Since $\int \chi_A d\nu=\int L_R(\chi_A)d\nu$ and $L_R(\chi_A)\leq \chi_A$ $\nu$-almost everywhere,  then $L_R(\chi_A)=\chi_A$  $\nu$-almost everywhere.  Hence $R^{-1}(A)=A$. 
 
 Now, if $$W=F(R) \setminus \left[ \{\textnormal{periodic points}\}  \cup \{\textnormal{rotational domains}\}\right]$$ then $A\cap W=\emptyset$. Otherwise, there is a point $z_0\in A\cap W$ and a disk $D$ centered at $z_0$ with $\nu(D)>0$ so that  $R^{-n}(D)$ forms a pairwise disjoint family of open subsets of $W$,  for $n=1,2,3,...$. But this contradicts that $\nu$ is invariant.    Since $R^{-1}(A\cap F(R))=A\cap F(R)$, then  $A\cap F(R)$ consists of periodic points and, by a normal families argument, contains at most two points $\{a,b\}$. Then $\omega$, which is the atomic part of $\nu$, consists of delta measures based on the set $\{a,b\}$.   
 
If $A\cap J(R)=\emptyset$, then we are done. Otherwise, $A\cap J(R)=J(R)$ and, by Theorem \ref{th.Lyubich}, the restriction of $\nu$ on $J(R)$ is a multiple of the measure of maximal entropy which finishes the proof.  
\end{proof}

\begin{theorem}\label{nonatomicLyubich}
Let  $\mathcal{F}$ be an admissable family of  rational maps. Assume $S=S(\mathcal{F})$   admits a right amenable Lyubich representation, then there exists a unique non-atomic Lyubich measure $m_S$. Even more, $m_S=m_s$ for any $s\in S$ with $deg(s)>1$ where $m_s$ is the measure of maximal entropy of $s.$
\end{theorem}

\begin{proof}
By Theorem \ref{th.RABranched}  and Lemma \ref{lm.atomic} there is a Lyubich measure $m_S$ given by a measure $\sigma=\delta_{y_0}$ for some $y_0\in \overline{\C}$. Since $\mathcal{F}$ contains a non-exceptional map $g$, then by Lemma \ref{lm.atomic}, the measure $m_S=\alpha(y_0) m_g+ \beta(y_0) \delta_{z_0}$, $\alpha(y_0)+\beta(y_0)=1$ for a suitable $z_0\in  \overline{\C}$ and    $m_g$ is the measure of maximal entropy of $g$. If $\beta(y_0)\neq 0$ the $z_0$ is determined by the condition $g^{-1}(z_0)=z_0$. 

If $\alpha(y_0)\neq0$ for a $y_0\in \overline{\C}$, we are done by Theorem \ref{th.Lyubich}, since $m_g$ is the claimed measure. 

Let us show that $\alpha(y_0)\neq 0$ for some $y_0$. Otherwise $\beta(y_0)=1$ for every $y_0\in \overline{\C}$ and $\mathcal{F}$ consists of polynomials, since by conjugation we can assume that $z_0=\infty$  because $s^{-1}(z_0)=z_0$ for every $s\in S$.  
Hence $\mathcal{F}$ satisfies the third case of the definition of admissability.

Thus, there exists a disk $D_\mathcal{F}$ such that $P(D_\mathcal{F})\subset D_\mathcal{F}$ for every $P\in S$. Let $K=\overline{\C}\setminus {D}_\mathcal{F}$ and take a compact subdisk $D\subset D_\mathcal{F}$ centered at $\infty$ and consider $\phi$ a continuous function such that $\phi(z)$ is $1$ for $z\in K$  and $0$  for $z \in D$.

Hence $$\int_{\overline{\C}} \phi(z) dm_S(z)=0.$$
Let  $x_0\in K$,  since $s^{-1}(K)\subset K$ for any $s\in S$, then, by construction of $\phi$, for every $s\in S$ we have   $$H(\phi)(s)=\int_{\overline{\C}}\rho(s)(\phi(z))d \delta_{x_0}(z)=(\rho(s)(\phi))(x_0)=1,$$
 where $\rho$ is the Lyubich representation. Thus $H(\phi)$ is  the constant function $1$, hence $M(H(\phi))=1$ for every mean $M$. Besides,  for a suitable $r$-invariant mean $L$ we have $$L(H(\phi))= \ell(\phi)=\int_{\overline{\C}}\phi(z) dm_S(z)=\phi(\infty)=0,$$ which is a contradiction.  \end{proof}

Let $R$ be a rational map with $deg(R)\geq 2$, then we define $E(R)$ to be the set of  all rational maps $Q$ such that $$L^*_Q(m_R)=m_R$$ where $L^*_Q$ is the dual operator to the Lyubich operator $L_Q$ and $m_R$ is the measure of maximal entropy of $R.$  The set $E(R)$ is a semigroup under composition.

Define $G(R)=E(R)\cap Mob$. Also, $G(R)$ can be characterized as the maximal subgroup of  $E(R).$

\begin{theorem}\label{th.entropy}
Let $P$ be a non-exceptional polynomial of degree at least 2, and $S$ be a subsemigroup of $E(P)$ then $S$ is right amenable. 
\end{theorem}
This theorem is  reciprocal of  Theorem \ref{nonatomicLyubich}. We start with the following. 

Let $U$ and $T$ be  semigroups with a homomorphism $\rho:T\rightarrow End(U).$
Let  $U\rtimes_\rho T$ denote the semidirect product of the semigroups $U$ and $T$ which is the set $U\times T$ equipped with the following multiplication
$$(u_1,t_1)\cdot (u_2,t_2)=(u_1\cdot \rho(t_1)(u_2),t_1t_2).$$ 

Next proposition puts together two results of M. Klawe (see \cite{Klawe}).

\begin{proposition}\label{prop.Klawe} Let $U$ and $T$ be semigroups and $\rho:T\rightarrow End(U)$ be a homomorphism.
\begin{itemize}
 \item If $U$ and $T$ are RA, then $U\rtimes_\rho T$ is RA.
\item If $U$ and $T$ are amenable semigroups and  $\rho(t)$ is surjective for every $t\in T$, then $U\rtimes_\rho T$ is amenable.
\end{itemize}

\end{proposition}

\begin{proof}
These appear as  Proposition 3.10  and Corollary 3.11 in \cite{Klawe}.
\end{proof}

\begin{theorem}\label{th.semidirect}
 Let $P$ be a polynomial map such that $E(P)$ is not abelian, then there are an abelian subsemigroup  $\Gamma < E(P)$ and a homomorphism $\Phi:\Gamma \rightarrow End(G(P))$ such that the semidirect product  $G(P)\rtimes_\Phi \Gamma$ is isomorphic to $E(P).$ Moreover, if $P$ is not M\"obius conjugated to $z^n$ then $\Gamma$ can be chosen as a cyclic semigroup. 
\end{theorem}

\begin{proof} 
First assume that $P$ is not affinely conjugated to $z^n$ for some $n.$ Since the elements of $G(P)$ leave the Julia set $J(P)$ invariant, then $G(P)$ is a  finite group of rotations around a common center (see for example Lemma 4 of \cite{BeardonSymmet}).

Let $Q\in E(P)$ be a non-linear polynomial of minimal degree. Since $E(P)$ is not abelian then by Atela-Hu theorem in \cite{AtelaHu} for every element  $\tilde{Q}\in E(P)$ there is $n\geq 0$ and  a $\gamma\in G(P)$ such that   $\tilde{Q}=\gamma \circ Q^n$.  But the degree of $Q$ is minimal and $Q\circ \gamma \in E(P)$ then $Q\circ \gamma=\gamma' \circ Q$. Thus the correspondence $\gamma\mapsto \gamma'$ defines a homomorphism $\Phi$ from  $\langle Q \rangle $ to $End(G(P))$. With $\Phi$,  construct the semidirect product $G(P)\rtimes_\Phi \langle Q \rangle$ so the correspondence $(u,g)\mapsto u\circ g$ defines a surjective isomorphism $G(P)\rtimes_\Phi \langle Q \rangle \rightarrow E(P).$  

Now assume $P$ is affinely conjugated to $z^n$ for some $n$. Then $E(P)$ is not abelian and after a suitable conjugation $G(P)$ becomes $G(z^n)$  which is generated by the group of all rotations around $0$ and the element $1/z$. In this situation, we can choose a subsemigroup $\Gamma \subset E(P)$ conjugated to all powers of $z$. Therefore $\Gamma$ is an abelian infinitely generated semigroup acting on $G(P)$ by semiconjugacy as a semigroup of surjective endomorphisms of $G(P)$. Again, the correspondence $(u,T)\rightarrow u\circ T$ generates an isomorphism  $G(P)\rtimes_\Phi \Gamma \rightarrow E(P).$
\end{proof}

As an immediate consequence we have.

\begin{corollary}\label{cor.ramenable}
 Let $P$ be a polynomial, then $E(P)$ is RA. Even more,  if $P$ is conjugated to $z^n$ then $E(P)$ is amenable.
\end{corollary}

\begin{proof} If $E(P)$ is abelian then it is amenable. Otherwise, the corollary follows from Proposition \ref{prop.Klawe} and Theorem \ref{th.semidirect}.
\end{proof}

Given a rational map $R$, let 
$$Deck(R)=\{\gamma \in Mob: R\circ \gamma=R\}$$ and $$Aut(R)=\{\gamma \in Mob: R\circ \gamma=\gamma \circ R\}.$$ 

\begin{corollary}\label{cor.GPGenerated}
 Let $P$ be a polynomial with finite  $G(P)$. Let $Q\in E(P)$ be a non-injective polynomial of minimal degree. Then there exist natural numbers $m$, $n$ such that $G(P)$ is generated by  $Deck(Q^m)$ and $Aut(Q^n).$
\end{corollary}

\begin{proof}
 
By Theorem  \ref{th.semidirect}, the map $Q$ defines an endomorphism $\Phi(Q):G(P)\rightarrow G(P)$ by the semiconjugacy $Q\circ \gamma = \Phi(Q)(\gamma)\circ Q$. Since $G(P)$ is finite the map $\Phi(Q)$ is surjective if and only if $\Phi(Q)$ is an automorphism of $G(P)$. In this case, there exists $n$ such that $\Phi^n(Q)=Id$ and thus $Q^n\circ \gamma=\Phi^n(Q)(\gamma)\circ Q=\gamma\circ Q^n.$

If $\Phi(Q)$ is not an automorphism,  then as $G(P)$ is finite there exists $k$ so that $\Phi$ acts on $\Phi^k(G(P))$ as an automorphism and $$G(P)\simeq Ker(\Phi^k)\otimes Im(\Phi^k).$$ Let $m>0$ be the minimal  number  satisfying the equation above. Let $n$ be the minimal natural number such that $\Phi^n:Im(\Phi^m) \rightarrow Im(\Phi^m)$ is identity. Then every $\gamma \in Im(\Phi^m)$ commutes with $Q^n$. On the other hand, if $\gamma\in Ker(\Phi^m)$ then $Q^m(\gamma)=Q^m$.
 \end{proof}

 \begin{example}\label{example} Let $P(z)=z^5+z^2$, then $$G(P)=\{\lambda z: \lambda^3=1\}$$
 and $$E(P)=\{\lambda^kP^l, \textnormal{ for } k=0,1,2, \textnormal{ and } l=0,1,...\}$$ is amenable by Theorem \ref{th.semidirect}. 
 
 Since $G(P)=Aut(P^2)$ then by Corollary \ref{cor.GPGenerated} we have  $m=0$, $n=2$. Therefore, the polynomial $Q=\lambda P$ commutes with $P^2$ but does not commutes with $P.$ 
 In particular, amenability does not implies commutativity even for polynomials. 
 \end{example}

\begin{lemma}\label{lm.polyRIM} Let $P$ be a polynomial with finite $G(P)$, then there
exists $\mu\in RIM(E(P)) $ so that $\mu(\chi_{\langle Q \rangle})>0$ for  every $Q\in E(P)$.
\end{lemma}

\begin{proof}
 We follow Klawe's proof of Proposition \ref{prop.Klawe} (Proposition 3.10 in \cite{Klawe}).
We summarize Klawe's construction of a RIM for a semidirect product $S=U\rtimes_\rho T$  
of RA semigroups $U$ and $T$ with representation $\rho:T\rightarrow End(U)$ as follows.  

\begin{itemize}
 \item Choose both a  RIM $\phi$ for $U$  and a RIM $\nu$ for $T$.
 \item For each $f \in L_\infty(S)$ construct the function  $\tilde{f}\in L_\infty(T)$ as follows:  given $a\in T$  let  $f_a(u)=f(u,a)$ for $u\in U$, then the family of functions $\{f_a\}$ belongs to $L_\infty(U)$. Let  $\tilde{f}(a)=\phi(f_a).$
 \item The mean $\mu$ on $S$  given by
 $$\mu(f)=\nu(\tilde{f})$$ is a RIM for $S.$
\end{itemize}

By Theorem \ref{th.semidirect}, the semigroup $E(P)$ contains a polynomial map $R$ such that $E(P)$ is isomorphic to the semidirect product of $G(P)$ and $\langle R \rangle$. Choose two RIMs $\phi$ and $\nu$ for $G(R)$ and $\langle R \rangle$, respectively. Since $G(R)$ is finite,  $\phi(\chi_A)\geq\frac{1}{|G(P)|}$ for each subset $A \subset G(P)$. Let $Q\in E(P)$  then there exists a number  $m\geq 0$ such that, for every $n>0$, $Q^n=\gamma_n R^{mn}$ with $\gamma_n\in G(P)$. 

If $f=\chi_{\langle Q \rangle}$ is the characteristic function of $\langle Q \rangle$ in $L_\infty(E(P))$, then 
the family of functions  $f_{R^n}(\gamma)=f(\gamma R^n)$ belongs to $L_\infty(G(P))$. Thus
the function  $\tilde{f}(R^n)=\phi(f_{R^n})\in L_\infty(\langle R\rangle )$. By construction 
$\tilde{f}\geq \frac{1}{|G(P)|}\chi_{\langle R^{m}\rangle}$.  Since $\nu$ is finitely additive and $r$-invariant, we conclude that $\mu(f)=\nu(\tilde{f})\geq \frac{1}{|G(P)|m}>0.$\end{proof}

Now we are ready to prove Theorem \ref{th.entropy}.

\begin{proof}[Proof of  Theorem \ref{th.entropy}]
Let  $Q\in S < E(P)$ with $Q\neq Id$. By Corollary 
\ref{cor.ramenable} the semigroup $E(P)$ is RA. By Lemma \ref{lm.polyRIM}, 
there exists $\mu\in RIM(E(P))$ such that $\mu(\chi_{\langle Q \rangle })>0$. 
Hence $\mu(\chi_S)\geq \mu(\chi_Q)>0.$ We finish the proof by  applying property 5 in Section \ref{sect.Semigroup}: 
\textit{Let $S_0$ be a subsemigroup of $S$, if there is  $\mu\in RIM(S) $ such that 
$\mu(\chi_{S_0})>0$, then $S_0$ is  right amenable itself.}
\end{proof}
The following Corollary implies the proof of Theorem \ref{th.rholyub}.
\begin{corollary}\label{cor.ramenability}
For an admissable  collection of polynomials $\mathcal{P}$ the following conditions are equivalent. 
 
 \begin{enumerate}
  \item The semigroup $S(\mathcal{P})$ is right amenable.
  \item The semigroup $S(\mathcal{P})$ is Lyubich right amenable.
  \item There exist $P\in \mathcal{P}$ such that $\mathcal{P}\subset E(P)$.

 \end{enumerate}

\end{corollary}

\begin{proof}

Part (1) implies (2) by  Theorem \ref{th.Day}.  Part (2) implies  (3) by Theorem \ref{nonatomicLyubich}. Part  (3) implies  (1) by Theorem \ref{th.entropy}.  
\end{proof}

\begin{theorem}\label{th.Intersection}
For an admissable collection  $\mathcal{F}$ of non-injective rational maps, the following affirmations are equivalent.
\begin{enumerate}
 \item The semigroup $S(\mathcal{F})$ is RA and $S(\mathcal{F})\times S(\mathcal{F}) \subset DIP$.
 \item $S(\mathcal{F})\times S(\mathcal{F})\subset AIP$.
\item The semigroup $S(\mathcal{F})$  is  RA and embeddable into a group.  
\end{enumerate}

\end{theorem}

To prove Theorem \ref{th.Intersection} we need the following several facts.

\begin{theorem}\label{th.Levinrel} Let $S=\langle s_1,...,s_m\rangle$ be a finitely generated right cancellative semigroup 
satisfying the Levin relations $s_i\circ s_j=s^2_i$ for all $i,j$.  Then

\begin{itemize}
 \item The semigroup $S$ is right amenable.
 \item If $S$ is left amenable then $S$ is cyclic.  
\end{itemize}

\end{theorem}

\begin{proof}

For the first point, the proof uses standard ergodic arguments. Let $m$ be the number of generators of  $S$. 

By Levin relations we have the following dichotomy for  any pair of elements $s_i$ and $s_j$ in the generating set.

Namely  either $$s_i=s_j$$ or $$\langle s_i \rangle\cap \langle s_j \rangle =\emptyset.$$ Indeed, if there are numbers $k,q$ such that $s_i^k=s_j^q$ then by the Levin relations  $k=q$ and $s_i^{k+1}=s_j^{k+1}$ but by right cancellation we have $s_i=s_j$.

Let $\mathcal{L}(s)$ be the length function that is the infimum of the length of $s$ as a word in the letters  $\{s_1,...,s_m\}$. Since every element of $S$ is the iteration of a   generator, then $\#\{s: \mathcal{L}(s)\leq n\}= mn+1.$

The semigroup $S$ acts by the right on $L_\infty(S)$. The spherical average $\Theta$ of this action is given by  $$\Theta_n(\phi)=\frac{1}{mn+1}\sum_{\mathcal{L}(s)\leq n}r_s(\phi).$$
 
Note that for each $n$ the operator $\Theta_n$ is positive with $\|\Theta_n\|=1$ and $\Theta_n(\chi_S)=\chi_S$, where $\chi_S$ is the characteristic function on $S.$ 

We claim that if $h$ is a generator of $S$ then $$\|r_h(\Theta_n(\phi))-\Theta_n(\phi)\| \leq \frac{2m\|\phi\|}{mn+1}$$ for every $\phi \in L_\infty(S).$ Indeed, by the Levin relations for the right action of $S$ we have $$r_h(\Theta_n(\phi))-\Theta_n(\phi)=\frac{1}{mn+1}\left[ \sum_{i=1}^m r^{n+1}_{s_i}(\phi)-\sum_{s_i\neq h} r_{s_i}(\phi)-\phi\right],$$ but the right action is a contraction, that is $\|r_h\|\leq 1$ so the claim follows.

Let $\mathcal{M}$ be an $L_1$ mean on $L_\infty(S)$, that is $\mathcal{M}$ is induced by a non-negative function $\omega:S\rightarrow \C$ with $L_1$-norm $\|\omega\|=\sum_{s\in S}\omega(s)=1$ and    $\mathcal{M}(\phi)=\sum_{s\in S} \phi(s)\omega(s)$. Consider the family of means $\mathcal{M}_n=\Theta^*_n(\mathcal{M})$ where $\Theta^*_n$ is the dual operator of $\Theta_n.$ Then $\mathcal{M}_n$ forms a precompact family in the $*$-weak topology. Note that $\mathcal{M}_n(\chi_S)=1$ since $\Theta(\chi_S)=\chi_S$, so we get that any accumulation  point of $\{\mathcal{M}_n\}$ is a mean. If $\mathcal{M}_0$  is an accumulation point then by the claim $\mathcal{M}_0$ is invariant by the right action of any generator of $S.$ Hence $\mathcal{M}_0\in RIM(S)$. Which finishes the first part of the proof.

The last part is by contradiction. Assume that $S=\langle s_1,...,s_m\rangle$ is left amenable with $m>1$ and $\{s_i\}_{i=1}^m$ forms a minimal set of generators. By the dichotomy above, $S=\bigsqcup_{i=1}^m \langle s_i \rangle.$ Take $L\in LIM(S)$,  since $L(\chi_S)=1$ there exists a generator $s_i$ such that $L(\chi_{\langle s_i\rangle})>0$. As $m>1$, let  $s_j$ be a generator  with $i\neq j$ and let $\Gamma=\langle s_i,s_j\rangle$ then $L(\chi_\Gamma)=L(\chi_{\langle s_i\rangle}+\chi_{\langle s_j\rangle})>0$. By Fact 5 of the Basic Facts of amenability,  the semigroup $\Gamma$ is a non-cyclic LA semigroup with two generators.  

Let us show that $\langle s_i\rangle\cap \langle s_j \rangle\neq\emptyset$. Indeed, if   $\langle s_i \rangle\cap \langle s_j \rangle=\emptyset$ then by Levin relations $s_j\circ s\in \langle s_j\rangle$ for every $s\in \Gamma$. Hence
$$l_{s_j}(\chi_{\langle s_i\rangle})(s)=\chi_{\langle s_i\rangle}(s_j\circ s)=0.$$ 
By left invariance $L(\chi_{\langle s_i \rangle})=L(l_{s_j}(\chi_{\langle s_i\rangle}))=0$. Which   contradicts the choice of $s_i$.  Hence, $\langle s_i\rangle\cap \langle s_j\rangle\neq \emptyset$. Therefore, by the dichotomy above, $s_i=s_j$ which is again a contradiction with the choice of $s_j$.  \end{proof}

\begin{theorem}\label{th.LIMRIM}
Let  $S=S(\mathcal{F})$ be an amenable semigroup  satisfying $RIM(S)\subset LIM(S)$, where $\mathcal{F}$ is an admissable collection of rational maps. Then, for every $f$ and $h$ in $S$ with $deg(f),deg(h)>1$ there are numbers $m,n$ satisfying $f^m=h^n$.
\end{theorem}

\begin{proof}
Fix two arbitrary elements $f, h\in S$ with $deg(f),deg(h)>1.$ Then by  Theorem \ref{nonatomicLyubich} the maps $f$ and $h$ have the same measure of maximal entropy, and by Theorem \ref{th.Levin-Przyticki}, there are iterates $F$ and $H$ of $f$ and $h$ respectively, satisfying the Levin relations. Let $\Gamma=\langle F, H \rangle $ be the semigroup generated by $F$ and $H$. By Theorem \ref{th.Levinrel}, the semigroup $\Gamma$ is RA. If $\Gamma$ is LA, then again by Theorem \ref{th.Levinrel}, $F=H$ and we are done. 

Let us show that indeed  $\Gamma$ is a left amenable semigroup. We follow a Theorem of Granirer (see Theorem E2 in \cite{Granirer}) aswell as the arguments of the proof of this theorem. The theorem states: 

\textit{Let $S$ be an LA semigroup with left cancellation and let $S_0\subset S$ be an LA subsemigroup. Then there is a linear isometry $T$ from the subspace of left invariant elements of $L_\infty^*(S_0)$ into the subspace of left invariant elements of $L_\infty^*(S)$ with $T(LIM(S_0))\subset LIM(S).$}

 More precisely, using the left cancellation and the left cosets of $S_0$ in $S$, Granirer constructs an isometric linear section $j:L_\infty(S_0)\rightarrow L_\infty(S)$ to the restriction $\rho:L_\infty(S)\rightarrow L_\infty(S_0)$, which is a positive linear map, so that  for every left invariant functional $\nu_0\in L_\infty^*(S_0) $ the following formula holds (page 55 of \cite{Granirer}).

$$T(\nu_0)(x)=\nu_0(\rho(x)),$$ for every $x\in j(L_\infty(S_0)).$ 

Let $s\in S_0$ and $r_s,\tilde{r}_s$ be the right action of $s$ on $L_\infty(S)$ and $L_\infty(S_0)$ respectively, then for every $x\in L_\infty(S)$ we have $$\rho(r_s(x))=\tilde{r}_s \rho(x).$$

Now assume $T(\nu_0) \in  RIM(S)$  and  $r_s( j(x_0))-j(\tilde{r}_s(x_0))\in ker (T(\nu_0))$  for  every $x_0\in L_\infty(S_0),$ then $\nu_0\in  RIM(S_0)$. Indeed
$$\nu_0(\tilde{r}_s(x_0))=T(\nu_0(j(\tilde{r}_s(x_0))))=T(\nu_0)(r_s (j(x_0)))=T(\nu_0)(j(x_0))=\nu_0(x_0).$$

Suppose  $T(\nu_0)\in RIM(S)$ then we  claim $$r_s (j(x_0))-j(\tilde{r}_s(x_0))\in ker (T(\nu_0)).$$ 
Indeed,  let $f=j(\chi_{S_0})$, then $$T(\nu_0)(f)=\nu_0(\rho(f))=\nu_0(\chi_{S_0})=1.$$

Moreover, $T(\nu_0)(\chi_{S})=1$ thus $\chi_S-f\in ker(T(\nu_0))$.  But $T(\nu_0)$ is a positive functional and $\chi_S-f\geq \chi_{S\setminus supp(f)},$ then $\phi\in ker(T(\nu_0))$ whenever $supp(\phi) \in S\setminus supp(f)$. Since $j$ is a positive isometric section of the restriction map  $\rho$, then for every
$s\in S_0$ and $x_0\in L_\infty(S_0)$ we have $$supp(r_s(j(x_0))-j(\tilde{r}_s(x_0)))\subset S\setminus supp(j(\chi_{S_0}))$$ as claimed.

By assumption $T(\nu_0)\in LIM(S)$ whenever
$\nu_0\in LIM(S_0)$. Hence, by the claim if $LIM(S)\subset RIM(S)$ then  $\nu_0\in RIM(S_0)$ and, in particular, 
$LIM(S_0)\subset RIM(S_0).$

To apply Granirer Theorem and the discussion above, we consider  $S^*$ to be the semigroup $S$ endowed with the antiproduct. Since $S$ is amenable with right cancellation and $RIM(S)\subset LIM(S)$, then $S^*$ is an amenable semigroup with left cancellation and $LIM(S^*)\subset RIM(S^*)$, hence $\Gamma^*\subset S^*$ is left amenable. It follows that  $\Gamma^*$ is an RA semigroup and hence $\Gamma$ is left amenable. This finishes the proof.
\end{proof}

Let us note that as a corollary we have the following statement.

\begin{corollary}\label{cor.Granierremark}
 Let $S$ be a right cancellative amenable  semigroup satisfying  $RIM(S)\subset LIM(S)$. Consider an RA subsemigroup $S_0<S$, then $RIM(S_0)\subset LIM(S_0).$
\end{corollary}

\begin{theorem}\label{th.cmiterate}
 Assume $\mathcal{F}\times \mathcal{F}\subset AIP$ for a non-exceptional collection of rational maps $\mathcal{F}$, then $S=S(\mathcal{F})$ is amenable.
\end{theorem}

\begin{proof}
The proof of the theorem is a consequence of a theorem due to M. Day \cite{DayAmenable} which states: 
\textit{Let $S=\bigcup S_n$,  where $S_n$ are  semigroups such that for every $m,n$ there exists $k$ with $S_m \cup S_n \subset S_k$. Then $S$ is amenable whenever the semigroups $S_n$ are amenable for every $n$.}

 Fix a non-exceptional element $Q\in \mathcal{F}$. For $n>0$, let $S_n$ be the subsemigroup of all  elements in $S$ commuting with $Q^n,$ then by assumption $S=\bigcup_n S_n$ and moreover for every finite collection of indexes  $n_1,...,n_k$
 there exists $N$ such that $\bigcup_{i=1}^k S_{n_i}\subset S_N$, for instance, take $N=\prod n_i$. To finish the proof we have to show that the semigroups $S_n$ are amenable.
 
Indeed, for each $n$ let $$M_n(Q)=\{\mu\in L^*_\infty(S_n):  l^*_{Q^n}(\mu)=r^*_{Q^{n}}(\mu)=\mu,\mu\geq 0, \mu(\chi_{S_n})=\|\mu\|=1 \}.$$ 

$M_n(Q)$  is a non-empty, closed, convex and compact set with respect  to the $*$-weak topology of $L_\infty^*(S_n).$  In fact, $M_n(Q)$ is a subset of the unit sphere in $L_\infty^*(S_n)$ so does not contains the zero element. 
Fix $n$, since every element  $s\in S_n$ commutes with $Q^n$ then  $r^*_s$ and $l^*_s$ leave $M_n(Q)$ invariant.  Hence we constructed left and right representations 
$\rho_l$ and $\rho_r$ of $S_n$ into the semigroup $End(M_n)$ of continuous endomorphisms of $M_n$. 

By Theorem \ref{th.RittEremenko} every element of $S_n$ shares a common iteration with $Q^n$, thus the images  $\Gamma_l:= Im(\rho_l)$ and $\Gamma_r:=Im (\rho_r)$  are  groups in $End(M_n)$. 

If $\Gamma_l$ and $\Gamma_r$ are finite groups then $S_n$ is amenable for every $n$. Indeed, as $M_n$ is convex the averages defined by  $A_r(\nu)=\frac{1}{\#\{\Gamma_r\}}\sum_{\gamma \in \Gamma_r} \gamma(\nu)$ leaves $M_n$ invariant.  This means that 
$A_r(M_n)=RIM(S_n)$. Similarly, if $A_l(\nu)=\frac{1}{\#{\Gamma_l}}\sum_{\gamma \in \Gamma_l} \gamma(\nu)$ then $A_l(M_n)=LIM(S_n).$ But we have $A^2_r=A_r$, $A_l^2=A_l$ and $A_r\circ A_l=A_l\circ A_r$ then $A_r\circ A_l(M_n)\subset LIM(S_n)\cap RIM(S_n)$. Thus $S_n$ is amenable. 

To finish, we need the following result of F. Pakovich (\cite{PakovichComRat}): 

Let $f$ be a non-exceptional rational map of degree at least $2$, let $C(f)$ be the semigroup of all rational maps commuting with $f.$
  Then there are finitely many rational maps
  $x_1,...,x_k\in C(f)$ such that every $g\in C(f)$ has the form $g=x_i\circ f^l$ for some $i$ and $l\geq 0.$ 
  
  Hence $\Gamma_r$  and $\Gamma_l$ belong to the image of a finite set of elements, so these groups are indeed finite. 
\end{proof}

In the proof of the previous theorem, it is enough that $\Gamma_l$ and $\Gamma_r$ are amenable. While the preparation of this work, Pakovich kindly inform us about his theorem in \cite{PakovichComRat} which significantly shortened our original proof of Theorem \ref{th.cmiterate}.

We are ready to prove Theorem \ref{th.Intersection}.

\begin{proof}[Proof of Theorem \ref{th.Intersection}] 
Let us show that (1) implies (2). Since
$S$ is RA, then by Theorem
\ref{th.Levin-Przyticki} and Theorem \ref{nonatomicLyubich} for every $P$ and $Q$ in $\mathcal{F}$ there are numbers $m$ and $n$ such that $P^m\circ Q^n=P^{2m}$ and   $Q^n\circ P^m=Q^{2n}$. Let us show that 
$P^m=Q^n.$ Indeed, the pair $(P^m,Q^n)$ satisfies the intersection property. So, there exist $z_0$ and sequences 
$k_i$, $l_i$ such that $P^{m k_i}(z_0)=Q^{n l_i}(z_0)$. 

First assume that
$k_i=l_i$, then by the Levin relations we have $$Q^n\circ Q^{n(k_i-1)}(z_0)=P^{m}\circ P^{m(k_i-1)}(z_0)=P^m\circ Q^{n(k_i-1)}(z_0)$$ So $P^m$ and $Q^n$ coincide on the infinite set $\{Q^{n(k_i-1)}(z_0)\}$, thus $P^m=Q^n.$  

If $k_i\neq l_i$, then again using the Levin relations we obtain $$P^{m(k_i+1)}(z_0)=P^m\circ Q^{n l_i}(z_0)=P^{m (l_i+1)}(z_0)$$ hence $z_0$ has a finite orbit, which contradicts that
$(P,Q)\in DIP$. For a non-exceptional $Q\in \mathcal{F}$ and for every $P\in \mathcal{F}$ there exists $n$ such that $P$ commutes with $Q^{n}$. So, for every element $R\in S$ there exists a number $m=m(R)$ such that $R$ commutes with $Q^m$. Thus $\langle R, Q^m \rangle$ is abelian and, by Theorem \ref{th.RittEremenko}, $R$ and $Q$ share a common iteration. Therefore, 
every pair of elements in $S$ also share a common iteration.

Theorem \ref{th.cmiterate} gives the implication from (2) to (1).

Now, let us show the equivalence of  (2) and (3). 

First let us show that (3) implies (2).
Since $S$ is RA, by Theorem \ref{th.Levin-Przyticki} and Theorem \ref{nonatomicLyubich}, if $R$ and $Q$ are non-identity elements in $S$, then there exist numbers
$n$ and $m$ such that $R^m$ and $Q^n$ satisfy the Levin relations but then $R^m=Q^n$ since $S$ is embeddable into a group.

(2) implies (3). First let us show that $S$ is a cancellative semigroup. We already know that $S$ is right cancellative, so let us show that is also left cancellative. Assume there exist three elements $A,X,Y$ in $S$ with  $AX=AY$. Then $P=XA$ and $Q=YA$ satisfy the Levin relations 
$$P\circ Q=P\circ P$$ $$Q\circ P=Q\circ Q.$$ By assumption there are numbers $m$ and $n$ such that $P^m=Q^n$ then $m=n$ since $deg (P)=deg (Q)$. Again by the Levin relations $P^{m+1}=Q^{m+1}$ then $$XA=P=Q=YA$$ and $X=Y$ by right-cancellation.  By Lemma \ref{lm.LIP}, $S$ has the left ideal intersection property.  Now we have fullfilled  the conditions of Ore Theorem  which finishes the proof.   \end{proof} 

\textbf{Remark}. Let us note that in the proof of Theorem \ref{th.Intersection} it is enough that  the intersection $\mathcal{O}_+(P^m, z_0) \cap \mathcal{O}_+(Q^n,z_0)$ is sufficiently large. For instance, 
if $$\#\{\mathcal{O}_+(P^m, z_0) \cap \mathcal{O}_+(Q^n,z_0)\}>deg(P^m)deg(Q^n)$$ the arguments still follow. Since bounds are invariants of the semigroup, it is interesting to find precise bounds on  the intersection of the orbits. 

Another conclusion that follows from Theorem \ref{th.Intersection} is that $r$-amenability is necessary to 
compare the intersection property with the algebraic property of sharing  a common iterate.  
As an immediate corollary we have.
\begin{corollary}
 Let $R$ be a non exceptional rational map  and let $S(R) $ be the semigroup of rational maps commuting with $R$, then $S(R)$ is an embeddable semigroup.  
\end{corollary}

\begin{proof} Let $\Gamma(R)=S(R)\cap Mob$ then $S(R)$ is 
 generated by $\Gamma(R)$ and $S(R)\setminus \Gamma(R).$ By Theorem \ref{th.RittEremenko} and Theorem \ref{th.cmiterate}, the semigroup $S(R)\setminus \Gamma(R) \subset AIP$ and 
 is amenable. Since $\Gamma(R)$ is a finite group, it is amenable. Hence $S(R)$ is amenable because it is the disjoint union of amenable semigroups. By Corollary \ref{cr.ratemb} it is enough to show that $S(R)$ is left cancellative.  But by 
 Theorem \ref{th.RittEremenko} the semigroup $S(R)\setminus \Gamma(R)$ is a semigroup satisfying the conditions of Theorem \ref{th.Intersection} part (2), so $S(R)\setminus \Gamma(R)$ is cancellative. 
 
 Finally, if  there exist $Q\in S(R)\setminus \Gamma(R)$ and $a,b\in \Gamma(R)$ with $a\neq b$ and $Q\circ a = Q\circ b.$ This implies that $Deck(Q)$ contains $a\circ b^{-1}$ which belongs to $\Gamma(R)$. Then 
 $Q$ and $ab^{-1}Q$ satisfies the Levin relations, since $Q$ and $ab^{-1}Q$ share a common iterate   then $ab^{-1}=Id$ which  is a contradiction.
\end{proof}

In the theorems above we  used the equation $X\circ A=X\circ B$ to study left-cancellation of the semigroups. The maps $P=A\circ X$ and $Q=B\circ X$ satisfy the 
Levin relations and, by Theorem \ref{th.Levin-Przyticki},  $P$ and $Q$ have the same measure of maximal entropy.  In \cite{PakovichEntropy},   Pakovich proved the reciprocal theorem, that is: 

If $P\circ Q=P\circ P$ and $Q\circ P=Q\circ Q$ then there are rational maps $X,A,B$ such that $P=A\circ X$ and $Q=B\circ X$ and $X\circ A=X\circ B.$ 

So, if $P$ and $Q$ satisfy the Levin relation then by Pakovich theorem above we arrive to the equation $X\circ A=X\circ B$.

\begin{definition}
 Given a  semigroup $S<Rat$ we say that $A\approx B$ if there exists $X\in G$ so that $X\circ A=X\circ B.$
\end{definition}

In general, relation $\approx$ is not an equivalence relation. However, for instance,  if $S$ satisfies the left ideal intersection property, then $\approx$ is an equivalence relation in $S.$

It is interesting to characterize when the equation $X\circ A=X\circ B$ defines an equivalence relation on semigroups of rational maps. 

\begin{proposition}\label{pr.33}
 Let $S=S(\mathcal{F})$ be an RA semigroup, where $\mathcal{F}$ is an admissable collection of rational maps. Then the relation $\approx$ defines an equivalence relation and the quotient semigroup 
 $S_1=S/_\approx$ is embeddable into a group. Even more, if 
 $\pi:S\rightarrow S_1$ is a projection homomorphism, then 
 for every $P,Q\in S$ with  $deg(P),deg(Q)\geq 2$  there are 
 numbers $m$ and $n$ so that $\pi(P)^m=\pi(Q)^n$. 
\end{proposition}

 \begin{proof}
  The proof that $\approx$ is an equivalence relation relies on 
  standard amenability arguments (see for example \cite{Klawe}).

 To verify that $\approx$ is an equivalence relation, it is enough to check transitivity.  Indeed, assume that there are $a,b,c\in S$ such that  $a\approx b$ and $b\approx c$, thus there are $f_1,f_2\in S$ with $f_1a=f_1b$ and $f_2b=f_2c$. By Lemma \ref{lm.LIP}, the semigroup $S$ has 
 the left ideal intersection property, so there are $\alpha, \beta\in S$ such that $\alpha f_1=\beta f_2$, thence 
 $$\alpha f_1 a =\alpha f_1 b=\beta f_2 b=\beta f_2 c=\alpha f_1 c.$$
 
 The multiplication induced over representative classes endows
 $S_1=S/_\approx$ with a semigroup product.  By Fact 6 of the Basic Facts of amenability,
 $S_1$ is an RA semigroup. In particular, $S_1$ has the left ideal intersection property.  To show that 
 $S_1$ is  embbedable into a group, by Ore Theorem we need to verify that $S_1$ is cancellative.
 First $S_1$ is $r$-cancellative since  $S$ is also $r$-cancellative and, by construction, $S_1$ is $l$-cancellative. 

 Finally, since $S$ is RA, by Theorem \ref{nonatomicLyubich} and Theorem \ref{th.Levin-Przyticki}, for every $P,Q\in S$ there are numbers $m,n$ such that $P^m$ and $Q^n$ satisfy the Levin relations. Since $S$ is embeddable into a group
 then $\pi(P^m)=\pi(Q^n)$ as claimed. \end{proof}
 
The following corollary produces, in the polynomial case,  a realization for semigroups   of the type of $S_1$ in the proposition above. 

 \begin{corollary}\label{cor.34} Let $S=S(\mathcal{F})$ be an RA semigroup, where $\mathcal{F}$ is an admissable collection of polynomials, then there exists a polynomial $P$  and an isomorphism $\phi:S/_\approx \rightarrow E(P).$  
 \end{corollary}

 \begin{proof}
  Since $S$ is  RA, by Theorem \ref{nonatomicLyubich}, Theorem \ref{th.entropy} and Theorem \ref{th.semidirect},  there exist a polynomial $P$, such that $S\subset E(P)$, and numbers $r,s$ and $t$ so that every element $Q\in S$ has the form $Q=\gamma \circ h \circ P^t$ where 
  $\gamma\in Deck(P^s)$ and $h\in Aut(P^r).$ Thus the 
  class $[Q]$ contains a unique element $h\circ P^t$ then 
the correspondence $[Q]\mapsto h\circ P^t$ induces the desired representation.
 \end{proof}

For arbitrary semigroups of rational maps, the previous corollary is still an open question.

For polynomials we have the following theorem.

\begin{theorem}\label{th.RLamenable}
 Let $\mathcal{F}$ be a non-exceptional family of polynomials with  $\mathcal{F}\times \mathcal{F}\subset AIP$. Then
 the semigroup $S(\mathcal{F})$ is amenable with $RIM(S(\mathcal{F}))\subset LIM(S(\mathcal{F})).$
\end{theorem}
\begin{proof}
 Let $P\in \mathcal{F}$ be non-exceptional, then by the conditions we have $S(\mathcal{F})\subset E(P)$. If $E(P)$ is abelian, then $S(\mathcal{F})$ is abelian and hence $S(\mathcal{F})$ is amenable with $RIM(S(\mathcal{F}))=LIM(S(\mathcal{F}))$. Otherwise,  by Theorem \ref{th.semidirect}, there is a polynomial $T\in E(P)$, a finite group $G(P)=E(P)\cap Mob$ and a representation $\rho: \langle T \rangle \rightarrow End(G(P))$ by semiconjugation so that $E(P)\cong  G(P)  \rtimes_\rho \langle T\rangle.$
 By Corollary \ref{cor.GPGenerated}, there are numbers $r,s$ such that the group $G(P)$ is the direct product of $K(P)=ker(\rho(T^r))$ and $A(P)=Aut(T^s).$
 Let $AE(P)$ be the subsemigroup of $E(P)$ generated by $T$ and $A(P).$

We claim that every subsemigroup $\Gamma$ in $AE(P)$ is amenable with 
 $RIM(\Gamma)\subset LIM(\Gamma)$. 
 
 To prove the claim. First,  
$$\sum_{\gamma\in A(P)} r_\gamma(\psi)=\sum_{\gamma\in A(P)}l_\gamma (\psi)$$ for every $\psi \in L_\infty(AE(P))$.

Indeed if $s\in AE(P)$ then $s=h\circ T^k$ for a suitable $h\in A(P)$ and $k\geq 0$. Since $\rho(T)$ is an automorphism of  $A(P)$, then $$\sum_{\gamma \in A(P)} r_\gamma(\psi)(s)=\sum_{\gamma\in A(P)} \psi(h\circ T^k \circ \gamma)$$
$$=\sum_{\gamma \in A(P)} \psi((\rho(T))^k(\gamma)\circ h\circ T^k)=\sum_{\gamma\in A(P)} l_{\gamma}(\psi)(h\circ T^k). $$  Then for every  $\phi\in RIM(AE(P))$, the averages $$A_r=\frac{1}{\#\{A(P)\}}\sum_{\gamma \in A(P)}r_\gamma $$ and $$A_l=\frac{1}{\#\{A(Q)\}}\sum_{\gamma \in A(P)}l_\gamma,$$ satisfy $$\phi=A_r^*(\phi)=A_l^*(\phi).$$ Since
$l^* _\gamma \circ A_l^*=A_l^*\circ l_\gamma^*=A^*_l$, we conclude that $l^*_\gamma \phi=\phi$ for every $\gamma \in A(P). $ In other words, every right invariant mean $\phi$ is invariant by the left action of $A(P).$ 

Second, let us show that $l^*_T(\phi)=\phi.$

For every $\psi \in L_\infty(AE(P))$ 
we have $$A_l(l_T(\psi))(s)=A_l(r_T(\psi))(s).$$ Indeed, let  $s=h\circ T^k$ then 
$$A_l(\psi(T\circ h\circ T^k))=A_l(\psi[(\rho(T))(h)\circ h^{-1}(h\circ T^k \circ T)])$$
$$=A_l(\psi(h\circ T^k \circ T))=A_l(r_T(\psi))(s).$$
 By duality and the fact that $\phi$ is left invariant under $A(P)$, we get $$l^*_T(A_l^*(\phi))=r_T^*(A_l^*(\phi))=\phi$$ which implies
 $l_T^*(\phi)=\phi.$
 Hence, by above $\phi\in LIM(AE(P)).$ 
 
  Third,  let $\Gamma< AE(P)$, then $\Gamma$ is right cancellative and by Corollary \ref{cor.ramenable} the semigroup $\Gamma$ is RA. Then by Corollary \ref{cor.Granierremark} we have $RIM(\Gamma)\subset LIM(\Gamma),$ as claimed.
 
 To finish the proof of the theorem, we have to show that $S(\mathcal{F})$ is isomorphic to a subsemigroup of $AE(P).$ 
 
 Let $Q\in \mathcal{F}$ be a polynomial of minimal degree.  Then $Q$ has the following expression.
 $$Q=h\circ \gamma\circ T^l$$ for a suitable $l\geq 1$ and $h\in K(P)$ and $\gamma\in A(P)$.  Fix $m\leq r$ such that $h\in Ker(\rho(T^m))$ and put 
 $$\tilde{h}=h\circ \rho(T^l)(h)\circ \rho(T^{2l})...\circ \rho(T^{(m-1)l})(h).$$ Then $\tilde{h}^{-1}\circ Q \circ \tilde{h}=\gamma\circ T^l.$ The family $\tilde{\mathcal{F}}=\tilde{h}^{-1}\circ \mathcal{F} \circ \tilde{h}$ generates a semigroup 
 $S(\tilde{\mathcal{F}})$ isomorphic to $S(\mathcal{F}).$
 Note that $\tilde{Q}=\tilde{h}^{-1}\circ Q \circ \tilde{h}=\gamma\circ T^{l}\in\tilde{\mathcal{F}}.$

 Now 
 let us show that $S(\tilde{\mathcal{F}})<AE(P).$ It is enough to show that $\tilde{\mathcal{F}}\subset AE(P)$. Otherwise, assume 
 that $\tilde{\mathcal{F}}$ contains a polynomial $R=\alpha\circ \beta\circ T^t$ for $t\geq 1$, $\alpha \in K(P)\setminus \{Id\}$  and $\beta \in A(P)$. By assumption there are numbers $d,e>0$ such that $R^e=\tilde{Q}^d$, hence $R$ commutes with $\tilde{Q}^d$.
 If $i=\#\{K(P)\}$ and $j=\#\{A(P)\}$ then for $k=ijd$ we have 
 $$R\circ \tilde{Q}^k=\alpha\circ \beta \circ \rho(T^l)(\gamma)\circ T^t\circ T^{kl}=\tilde{Q}^k\circ R=\gamma \circ \beta \circ T^{kl}\circ T^{t}.$$
 Then $\alpha\in A(P)$ which is a contradiction by Corollary \ref{cor.GPGenerated}. \end{proof}

\begin{theorem} \label{th.familypolynomials}
Given an admissable  collection of polynomials  $\mathcal{F}$, the following statements are equivalent.
\begin{enumerate}
 \item $\mathcal{F}\times \mathcal{F}\subset DIP$.
 \item $\mathcal{F}\times \mathcal{F}\subset AIP$.
 \item  $S(\mathcal{F})\times S(\mathcal{F})\subset DIP$.
 \item $S(\mathcal{F})\times S(\mathcal{F})\subset AIP$.
 \item $S(\mathcal{F})$ is amenable with $RIM(S(\mathcal{F}))\subset LIM(S(\mathcal{F})).$
 \item The semigroup $S(\mathcal{F})$ is embeddable into a virtually cyclic group.
\end{enumerate}

\end{theorem}

\begin{proof}

 By the Ghioca-Tucker-Zieve Theorem in \cite{GhiTucZieve}  (1) is equivalent to (2) and (3) is equivalent to (4). Clearly (3) implies (1).  By Theorem \ref{th.LIMRIM} (5) implies (4). Finally, by Theorem \ref{th.RLamenable} (2) implies (5). 
 
 Now (4) implies (6). By Theorem \ref{th.Intersection} the semigroup $S(\mathcal{F})$ is embeddable into a group. As in the proof of Theorem \ref{th.RLamenable} there exist a polynomial $T$ and a finite group $A(T)\subset Mob$ so that $T$ acts on $A(T)$ by semiconjugacy and generates a representation $h:\langle T \rangle \rightarrow Aut(A(T)).$   The semigroup $AE(T)=\langle T,A(T)\rangle\cong A(T)\rtimes_h \langle T \rangle$  contains an isomorphic copy  of $S(\mathcal{F}).$ Let us show that $S(\mathcal{F})$ is a subsemigroup of a virtually cyclic group.
 First note that $AE(T)\cong A(T)\rtimes_{\tilde{h}}\mathbb{N} $ where $\tilde{h}(n)=h(T^n)\in Aut(A(T)).$ Since $\tilde{h}(n)$ is an automorphism we can extend $\tilde{h}$ on negative integers by the formula $$\tilde{h}(-n)=(h(T^n))^{-1}.$$ Hence 
 $A(T)\rtimes_{\tilde{h}} \mathbb{N}\subset A(T)\rtimes_{\tilde{h}}\mathbb{Z} $. But $A(T)\rtimes_{\tilde{h}}\mathbb{Z} $ is a semidirect product of a cyclic group with a finite group, so it is virtually cyclic, then $AE(T)$ is the positive part of a  virtually cyclic group.
 
Now (6) implies (4).  Assume that $S(\mathcal{F})$ is embeddable into a virtually cyclic group $\Gamma$, and $\tau$ be the generator of the corresponding cyclic subgroup of finite index. Let $T$ be an element in $S(\mathcal{F})$ corresponding to $\tau$, let $P\in S(\mathcal{F})$ of degree at least $2$ and $p\in \Gamma$ the corresponding element. Then $p$ is an element of infinite order, so there exists $k$ with  $p^k\in \langle \tau \rangle$, hence $(P,T)\in AIP.$
 \end{proof}

 Now we proceed to the proof of Theorem \ref{th.Ruelle}, for which we devote the last section.
 
 \section{Left amenability of Ruelle representation}

 We begin with the following observation.
 We say that a semigroup $S<Rat$ is \textit{deformable} if there exists $f:\overline{\C}\rightarrow \overline{\C}$ a quasiconformal homeomorphism so that $S_f=f\circ S \circ f^{-1}<Rat$ and $S_f$ is not M\"obius conjugated to $S$. 
 
  \begin{proposition}\label{pr.hypstruc}
 Let $S$ be an RA semigroup of non-injective rational maps. If  $S$  contains a hyperbolic structurally stable map then $S$ is deformable.
 \end{proposition}

 \begin{proof}
  Let $R$ be a hyperbolic structurally stable element of $S$, then by Theorem \ref{nonatomicLyubich}   for every $Q\in S$ with $deg(Q)>1$ we have $J(Q)=J(R)$. By the Levin relations we have
  \begin{itemize}
   \item Every $Q\in S$ with $deg(Q)>1$ is hyperbolic.
   \item For every periodic component $D$ in the Fatou set $F(R)=\overline{\C}\setminus J(R)$ we have  $$Q^{-1}(\mathcal{O}_-(R,D))=\mathcal{O}_-(R,D)$$ for every $Q\in S$ and where $\mathcal{O}_-(R,D)=\bigcup_{n=0}^\infty R^{-n}(D).$
  \end{itemize}
  
Let $K_D:\mathcal{O}_-(R,D)\rightarrow \C$ be the K\"onig linearizing function $D$ for $R$ in $\mathcal{O}_-(R,D)$, so $K_D(R)=\lambda K_D$ for some multiplier $\lambda$. Note that $K_D$ also linearizes every $Q\in S$ with $deg(Q)>1$.  Indeed, by the Levin relations we have numbers $m$ and $n$ so that $R^n\circ Q^m=R^n\circ R^n$ and $Q^m\circ R^n=Q^m\circ Q^m$ then $K_D(Q^{m})=\lambda^n K_D$.  Then the differential  $\mu=\frac{\overline{K'_D}K_D}{\overline{K_D}K'_D}\frac{d\overline{z}}{dz}$ is invariant for $Q$ and $R$,  so $\mu$ is an invariant Beltrami differential for every element in $S$. Thus, for $t\in (0,1)$, let $g_t$ be the quasiconformal map with Beltrami coefficient $t\mu$. Then $g_t$ defines a non-trivial deformation for $S.$
 \end{proof}

 \textbf{Remark.}  Let us note the following curious fact, if the semigroup $S$ is quasiconformally deformable with Beltrami differential $\mu$ such that $\overline{supp(\mu)}\neq \overline{\C}, $ then for all $R,Q\in S$ with $deg(R),deg(Q)>1$ we have $J(R)=J(Q)$. 
 
 For a subclass of RA semigroups we can say more. Let $S<Rat$ be a semigroup and let 
 $\phi:S\rightarrow Rat$ be an monomorphism preserving the degree, that is  $deg(\phi(Q))=deg(Q)$ for every $Q \in S$. We will say that $S$ is structurally stable if every monomorphism preserving degree $\phi:S \rightarrow Rat$, which is sufficiently close to the identity on generators, is generated by a quasiconformal
 homeomorphism of $\overline{\C}.$
 \begin{proposition}\label{pr.collstruc}
  Let $\mathcal{F}=\{R_i\}$ be a finite collection of rational maps of degree at least $2$ such that  $R_i\circ R_j=R^2_i$ for every pair $i,j$. Then the semigroup $S(\mathcal{F})$ is  structurally stable whenever $S(\mathcal{F})$ contains a structurally stable map. 
 \end{proposition}
\begin{proof}
 If $g\in S(\mathcal{F})$ is structurally stable, then $g$ is indecomposable and therefore is one of the generators, say $R_1$. If $\phi:S(\mathcal{F})\rightarrow Rat$ is a representation sufficiently closed to the identity representation, there exist a quasiconformal homeomorphism $f:\overline{\C}\rightarrow \overline{\C}$ such that $\phi(R_1)=f\circ R_1 \circ f^{-1}$. 
 
 We claim that $\phi(Q)=f\circ Q \circ f^{-1}$ for every $Q\in S(\mathcal{F}).$
 
 It is enough to check the latter equality holds for the generators $R_i$. If $\mu=\frac{\bar{\partial}f}{\partial f}$, then $\mu$ is invariant for all generators by the Levin relations. 
 
First let us assume that $R_i=\gamma_i\circ R_1\circ \gamma_i^{-1}=\gamma_i\circ R_1$ and $\phi(R_i)=h_i\circ \phi(R_1)\circ h_i^{-1}$ for suitables $h_i\in Deck(\phi(R_1))$ and  $\gamma_i\in Deck(R_1)$, respectively.  Since $\gamma_i$ 
leaves $\mu$ invariant then $Deck(\phi(R_1))=f\circ Deck(R_1)\circ f^{-1}$.  If the semigroup $T=\langle S(\mathcal{F}), Deck(R_1)\rangle$ then $\phi(S(\mathcal{F}))\subset f\circ T \circ f^{-1}\subset Rat.$

As $\phi(R_i)$ is close to $R_i$ for all $i$, it follows that $h_i$ is close to $\gamma_i$ and $\gamma_i$ is close to $f\circ \gamma_i \circ f^{-1}$, as $f$ is close to the identity too.

Therefore, $h_i$, $f\circ \gamma_i \circ f^{-1}\in Deck(\phi(R_1))$ are sufficiently close and hence coincide since $Deck(\phi(R_1))$ is discrete. In conclusion, $\phi(S(\mathcal{F}))=f\circ S(\mathcal{F}) \circ f^{-1}$ as claimed.

It remains to show that  $R_i=\gamma_i\circ R_1\circ \gamma_i^{-1}=\gamma_i\circ R_1$ and $\phi(R_i)=h_i\circ \phi(R_1)\circ h_i^{-1}$ for suitables $h_i\in Deck(\phi(R_1))$ and  $\gamma_i\in Deck(R_1)$  for every $i.$

Since $R_1\circ R_i=R_1^2$  hence $Deck(R_i)\subset Deck(R_1)$ then by Theorem \ref{th.RittEremenko} the maps $R_1$ and $R_i$ share a common right factor, that is there are rational maps $X,Y$ and $W$ such that $R_1=X\circ W$ and $R_i=Y\circ W$. But $R_1$ is indecomposable  then $Deg(X)=Deg(Y)=1$. It follows that  $X\circ Y^{-1}\in Deck(R_1)$ and $R_i=Y\circ X^{-1}\circ R_1$.  Finally, the map $\phi(R_1)$ is structurally stable  as a quasiconformal deformation of a structurally stable map, so it is also indecomposable. Now we can repeat the arguments for $\phi(R_1).$ 

\end{proof}

Therefore, a semigroup $S$ satisfying the Levin relations possesses an non-zero invariant Beltrami differential if and only if there is an element of $S$  possessing an invariant Beltrami differential. 

In what follows, for every rational map $R$ and a every completely invariant set $A\subset \overline{\C}$ of positive Lebesgue measure,  we construct a semigroup of operators satisfying the Levin relations and acting on  $L_1(A)$ and show that the action is left amenable whenever $R$ does not admits a non-zero Beltrami differential supported on $A.$

 \begin{definition} Let $R$ be a rational map.
 Let $\sigma$ be an analytic arc in $\C$ containing all critical values of $R.$ 
Let $U=\overline{\C}\setminus \sigma$ and $D=R^{-1}(U)$, then $D=\bigcup_{i=1}^{deg (R)} D_i$ and $\pi_1(D_i)=1$ and $R:D_i\rightarrow U$ is holomorphic homeomorphism.
Set $R_i=R|_{D_i}$ and for each $i,j$ define  the piecewise conformal map $$h_{i,j}=\begin{cases}
    R_j^{-1}\circ R_i,& \text{on } D_i\\
      R_i^{-1}\circ R_j,& \text{on } D_j\\
    Id,              & \text{otherwise.}
\end{cases}$$
Then $h_{i,j}$ is a  piecewise conformal almost everywhere bijection such that $h_{i,j}^2=Id$ and $h_{i,i}=Id$ everywhere.  We denote by  $D(R)$ the group  generated by the maps $h_{i,j}$ as the full deck group of $R$ associated to the arc $\sigma.$

\end{definition}

Note that $D(R)$ is isomorphic to the symmetric group on $deg(R)$ symbols. For every $\gamma\in D(R)$ we have that $R(\gamma)=R$ almost everywhere. The group $D(R)$ acts on  $L_1(\mathbb{C})$ by the push-forward map $$\gamma_*:f\rightarrow f(\gamma)\gamma'^{2}$$with $\|\gamma_*\|_{L_1}\leq 1$ for every $\gamma\in D(R)$.
 
 For every subgroup $\Gamma <D(R)$ and $\gamma\in \Gamma$, let $R_\gamma=\gamma \circ R\circ \gamma^{-1}=\gamma \circ R$. Define the semigroup 
 $$\mathcal{S} (\Gamma)=\langle R_\gamma \rangle_{\gamma\in \Gamma}.$$ Then $\mathcal{S}(\Gamma)$ is a finitely generated semigroup of piecewise holomorphic maps which is RA by Theorem \ref{th.Levinrel}. For example if $\Gamma <Deck(R)$ then $\mathcal{S}(\Gamma)$ consists of rational maps.  
 
 Using the action of $R$ on $L_1(\C)$ by the Ruelle operator $R_*$ we construct the Ruelle representation $\rho:\mathcal{S}(D(R))\rightarrow End(L_1(\overline{\C}))$  defined by the formulas on generators: $$\rho(R_\gamma)(\phi)=(R_\gamma)_*(\phi)=\gamma_*\circ R_*(\phi)$$ for $\phi\in L_1(\C).$  If $A$ is a completely invariant positive Lebesgue measure set, that is $Leb(R^{-1}(A)\setminus A)=0$, then $Leb(\gamma(A)\setminus A)=0$ for $\gamma \in D(R)$, where $Leb$ denotes the Lebesgue measure. 
 
 \begin{proposition}\label{pr.BeltDR}
  Let $R$ be a rational map and $A$ be a completely invariant set of positive Lebesgue measure.  Assume that $A$ does not support a non-zero invariant  Beltrami 
  differential of $R$, then the Ruelle representation of $S(D(R))$ on $L_1(A)$ is left amenable.
 \end{proposition}
 
 \begin{proof}
  The semigroup $S(D(R))$ is RA by Theorem \ref{th.Levinrel}, then  the space  $X_{\rho}\subset L_\infty(S(D(R)))$ possesses a right-invariant mean $m$. Recall that $X_{\rho}$ is the closure of the linear span of constant functions  together with the space $Y_{\rho}$. 
  
  We claim that $ker(m)$ contains $Y_{\rho}.$ 
  
  Otherwise, there are two elements $\psi\in L_1(A)$ and $\nu\in L_\infty(A)$ so that $m(\phi_{\psi,\nu})\neq 0.$ Then 
  $$M(f)=m(\phi_{ f,\nu})$$ is a continuous $R_*$-invariant functional on $L_1(A).$  But $M(\psi)=m(\phi_{\psi,\nu})\neq 0$ then by the Riesz representation theorem there exists an invariant Beltrami differential $\mu\neq 0$ which is a contradiction.

  Since $X_{\rho}$ and $Y_{\rho}$ are both left-invariant then by the claim every right mean on $X_{\rho}$ is  left invariant.
  
 \end{proof}

 Conversely we have the following theorem. 
 
 \begin{theorem}\label{th.fixed}
  Let $R$ be a rational map and $\Gamma<D(R)$ be a transitive subgroup. Assume that $\mathcal{S}(\Gamma)$ is $\rho$-LA, where $\rho$ is the Ruelle representation. Then the following conditions are equivalent.

  \begin{enumerate}
   \item $R_*$ has non-zero fixed points in $L_1(\C).$
   \item $R$ is M\"obius conjugated to a flexible Latt\'es map.
  \end{enumerate}

 \end{theorem}
 
 \begin{proof}
  (1) implies (2). Assume that $R_*$ has a non-zero fixed point $f\in L_1(\C).$ Then by Lemma A in \cite{MakRuelle}, there exists an invariant Beltrami differential $\mu$ with $\mu=\frac{|f|}{f}$ almost everywhere on the support of $f.$  We can assume  that $R$ acts ergodically on the support of $\mu$.  Then the representation $R_*:L_1(supp(\mu))\rightarrow L_1(supp(\mu))$ has fixed point $\alpha\neq 0$ if and only if $\alpha$ is a multiple of $f$. Even more, the Beltrami operator $(R_*)^*:L_\infty(supp (\mu)) \rightarrow L_\infty( supp (\mu))$ has a fixed point $\beta\neq 0$ if and only if $\beta$ is a multiple of $\mu.$ 
  Then by the separation principle, we conclude that $R_*:L_1(supp (\mu))\rightarrow L_1(supp (\mu))$ is mean-ergodic. 
 
 Even more $R_*$ is weakly almost periodic. Indeed since $R_*$ is mean-ergodic then the conjugated operator $T(\phi)=\mu R_*(\overline{\mu}\phi)$ is also a mean-ergodic operator with the same norm. A straightforward computation shows

 $$T(\phi)(y)=\sum_{R(x)=y} \frac{\phi(x)}{|R'(x)|}=\sum \phi(\zeta_i(y))|\zeta'_i|^2(y)$$ is a positive operator which is almost weakly periodic by Theorem \ref{th.Lin}, where $\zeta_i$ is a complete local system of branches of $R^{-1}$. So  $R_*$ is weakly almost periodic on $L_1(supp(\mu))$. 
 
The semigroup $\mathcal{S}(\Gamma)$  consists only of iterations of the generators and every generator is conjugated to $R$. Hence $\rho(\mathcal{S}(\Gamma))$ also consists only of iteration of the generators $\rho(R_\gamma)$ and each $\rho(R_\gamma)$ is conjugated to $\rho(R)$ where $\rho$ is the Ruelle representation. This implies that $\rho(\mathcal{S}(\Gamma ))$ is a weakly almost periodic semigroup  of operators on $L_1(supp( \mu)).$
 
 Since $\mathcal{S}(\Gamma)$ is $\rho$-LA, we claim that there exist a functional $\ell\in L^*_\infty(supp (\mu))$ which is invariant for the semigroup $(\rho(\mathcal{S}(\Gamma)))^*=\{t^*: t\in \rho(\mathcal{S}(\Gamma))\}$. 
 
 Indeed if $L$
 is a mean we define the functional $$\ell(h)=L(\phi_{h,f})$$
for  $\phi_{h,f}\in L_\infty(\mathcal{S}(\Gamma))$ given by  $$\phi_{h,f}(g)=\int h \rho(g)(f)|dz|^2$$ where $g\in \mathcal{S}(\Gamma)$, $h\in L_\infty(supp (\mu))$ and $f\in L_1(supp(\mu))$. Since $L$ is left invariant we get $\ell(t^*(h))=\ell(h)$ for every $t\in \rho(\mathcal{S}(\Gamma))$. 

Now we continue the proof of the theorem by standard arguments of functional analysis (see for example \cite{DunfordSchwartz}). The functional $\ell$ generates a finite complex valued invariant finitely additive measure $\alpha_\ell$
defined by the formula
$$\alpha_{\ell}(A)=\ell(\chi_A)$$ where $A$ is a measurable subset of $supp(\mu).$
From the definition follows that $\alpha_\ell$ is null on every zero Lebesgue measure subset of $supp(\mu).$ Next we
show that $\alpha_\ell$ is a measure absolutely continuous with respect to Lebesgue. It is enough to show that $\alpha_\ell$ is a countably additive set function. That is $$\alpha_\ell(\bigcup A_i)= \sum\alpha_\ell(A_i)$$ for every pairwise disjoint family of measurable subsets of $supp(\mu).$

Since $\rho(\mathcal{S}(\Gamma))$ is weakly almost periodic then for every $\epsilon>0$ and every $\beta\in L_1(supp(\mu))$  there exists a $\delta>0$ such that 
$$\int_B |t(\beta)|\leq \epsilon$$ for every $t\in \rho(\mathcal{S}(\Gamma))$ whenever the 
Lebesgue measure of $B$ is less than $\delta.$

Let $X\subset supp(\mu)$ a finite Lebesgue measure set which has a decomposition $X=\bigcup_{i=0}^ \infty  A_i$ by a family of pairwise disjoint measurable subsets. Then for every $k$ we have 
$$\alpha_\ell(X)=(\sum_{i=0}^k \alpha_\ell(A_i))+\alpha_\ell(\bigcup_{i=k+1}^\infty A_i))$$ by finite additivity.

Since $\rho(S(\Gamma))(f)$ is a weakly precompact set, for every  $\epsilon>0$  we get a  $\delta>0$ so that if  $k_0$ is such that $Leb(X_k)<\delta$ for $k>k_0$,  where $X_k=\bigcup_{i=k+1}^\infty A_i$, then 

$$|\alpha_\ell(X_k)|\leq |L(\phi_{\chi_{X_k},f})|\leq \sup_{g\in \mathcal{S}(\Gamma)} \int_{X_k}|\rho(g)(f)||dz|^2\leq \epsilon.$$

Then $\alpha_\ell$ is a finite measure which is absolutely continuous with respect to the Lebesgue measure on $supp(\mu)$. Hence, there exists a non-zero $w\in L_1(supp(\mu))$ so that $\ell(h)=\int hw |dz|^2$.
Since $\ell$ is $\rho(\mathcal{S}(\Gamma))^*$ invariant, then 
$w$ is $\rho(\mathcal{S}(\Gamma))$ invariant and therefore $w$ is a multiple of $f.$

As $R_*(f)=f$, we conclude that $f$ is a fixed point for $\gamma_*$, with $\gamma\in \Gamma.$ Since $\Gamma$ is transitive, we can choose $d=deg(R)$ elements $\gamma_1,...,\gamma_d\in \Gamma$ so that for every fixed  branch $\zeta_i$ of $R^{-1}$ on $\overline{\C}\setminus \sigma$ we have that the collection $\{\gamma_j\circ \zeta_i\}$ forms a complete collection
of branches of $R^{-1}$ on $\overline{\C}\setminus \sigma.$  
Therefore, $$dw(\zeta_i)(\zeta'_i)^2=\sum_{j} (\gamma_j)_*(w)\circ (\zeta_i)(\zeta'_i)^2=\sum w(\zeta_j)(\zeta'_j)^2=w.$$ Then for every $z\in R^{-1}(\overline{\C}\setminus \sigma)$ we have 
$$\frac{w(R(z))R'^2(z)}{deg(R)}=w(z).$$

Hence $|w|$ defines a continuous functional on $C(\overline{\C})$, via $\phi \mapsto \int_\C \phi |w| |dz|^2$, which is invariant under Lyubich operator $L_R$  and so by Theorem \ref{th.Lyubich}  and Lemma \ref{lm.atomic} it is  the density of the measure of maximal entropy for $R.$ Thus, the map $R$ has maximal entropy measure absolutely continuous with respect to Lebesgue. By Zdunik's Theorem (see \cite{ZdunikParab}) the map $R$ is an exceptional map. Since $\frac{|f|}{f}$ is an invariant Beltrami differential for $R$, then $R$ is a  flexible Latt\`es map. 

Now (2) implies (1). If $R$ is a flexible Latt\`es map, then again by Zdunik's Theorem, the measure $m_R$ of maximal entropy of $R$ is absolutely continuous with respect to Lebesgue. Then $dm_R=\omega|dz|^2$ with $\omega\in L_1(\overline{\C})$, $\omega>0$  and $$\frac{\omega(R)|R'|^2}{deg(R)}=\omega$$ almost everywhere. On the other hand, $R$ has non-zero invariant Beltrami differential $\mu$, hence the function $\overline{\mu}\omega$ is fixed by $R_*$, and we are done. 
\end{proof}

  \bibliographystyle{amsplain}
\bibliography{workbib}

\Addresses
\end{document}